\documentclass[10pt,twoside]{amsart}
\usepackage{amsmath,amsfonts,amssymb,amsxtra,setspace,xspace,graphicx,lmodern,psfrag,epsfig,color,latexsym}
\usepackage[T1]{fontenc}
\usepackage[colorlinks=true]{hyperref}\hypersetup{urlcolor=blue, citecolor=red}

\usepackage{caption}

\newtheorem{thm}{Theorem}
\newtheorem{lem}[thm]{Lemma}
\newtheorem{cor}[thm]{Corollary}
\newtheorem{prop}[thm]{Proposition}
\newtheorem{rem}[thm]{Remark}
\newtheorem{defi}[thm]{Definition}

\newcommand{\R}{{\mathbb R}}

\def\eps{\varepsilon}
\def\twoscale{\stackrel{2-scale}{\relbar\joinrel\relbar\joinrel\relbar\joinrel\relbar\joinrel\relbar\joinrel\rightharpoonup}}

\begin{document}

\title[Heterogeneity and strong competition in ecology]{Heterogeneity and strong competition in ecology}

\author[Harsha Hutridurga]{Harsha Hutridurga}
\address{H.H.: Department of Mathematics, Imperial College London, London, SW7 2AZ, United Kingdom.}
\email{h.hutridurga-ramaiah@imperial.ac.uk}

\author[Chandrasekhar Venkataraman]{Chandrasekhar Venkataraman}
\address{C.V.: School of Mathematics and Statistics, North Haugh, St Andrews KY16 9SS, United Kingdom.} 
\email{cv28@st-andrews.ac.uk}

\date{\today}

\begin{abstract}
We study a competition-diffusion model while performing simultaneous homogenization and strong competition limits. The limit problem is shown to be a Stefan type evolution equation with effective coefficients. We also perform some numerical simulations in one and two spatial dimensions that suggest that oscillations are detrimental to invasion behaviour of the species.
\end{abstract}

\maketitle


\setcounter{tocdepth}{1}
\tableofcontents

\thispagestyle{empty}

%%%%%%%%%%%%%%%%%%%%%%%%%%%%%%%%%%%%%%%%%%%%%%%%%%%%%%%%%%%%%%%%%%%%%%
%%%%%%%%%%%%%%%%%%%%%%%%%%%%%%%%%%%%%%%%%%%%%%%%%%%%%%%%%%%%%%%%%%%%%
\section{Introduction}
%%%%%%%%%%%%%%%%%%%%%%%%%%%%%%%%%%%%%%%%%%%%%%%%%%%%%%%%%%%%%%%%%%%%%%
%%%%%%%%%%%%%%%%%%%%%%%%%%%%%%%%%%%%%%%%%%%%%%%%%%%%%%%%%%%%%%%%%%%%%%
This article attempts to understand the effect of rapid oscillations in the diffusion coefficients in strong competition limits of models from theoretical ecology. In particular, the current work provides a starting point towards understanding the role played by heterogeneous mobility on invasion behaviour in  models for competition between two motile species.  For analytical tractability, we consider only periodic oscillations in the diffusivities. Take $Y:=[0,1)^d$ to be the unit reference periodicity cell. Consider bounded positive definite matrices $A, B\in\mathrm L^\infty(Y;\R^{d\times d})$, i.e., there exists a positive constant $\beta$ such that
\begin{align*}
\zeta^\top\!\! A(y) \zeta \ge \beta \left\vert \zeta \right\vert^2,
\qquad
\zeta^\top\!\! B(y) \zeta \ge \beta \left\vert \zeta \right\vert^2
\quad
\mbox{ for a.e. }y\in Y
\mbox{ and }\forall\zeta\in\R^d,
\end{align*}
where we use $\cdot^{\top}$ to denote the transpose. As is classical in the periodic homogenization setting, we extend $A$ and $B$ to the full-space by $Y$-periodicity, i.e.,
\begin{equation*}
\begin{array}{cc}
A(x) = A(x+k\, {\bf e}_i)
\\[0.2 cm]
B(x) = B(x+k\, {\bf e}_i)
\end{array}
\quad
\mbox{ for a.e. }x\in \R^d
\mbox{ and }\forall k\in\mathbb{Z}^d
\quad \forall i\in\{1,\cdots,d\},
\end{equation*}
with $\{{\bf e}_i\}_{i=1}^d$ denoting the canonical basis in $\R^d$. Taking $0<\eps\ll 1$ to be the heterogeneity length scale, we define highly oscillating diffusion coefficients as
\begin{align*}
A^\eps(x) := A\left(\frac{x}{\eps}\right);
\qquad
B^\eps(x) := B\left(\frac{x}{\eps}\right)
\qquad
\mbox{ for }x\in\Omega.
\end{align*}
Note that smaller the parameter $\eps$ is, greater will be the oscillations in $A^\eps$ and $B^\eps$ defined above. We would like to understand the dynamics of competition-diffusion in the context of two populations, say $\mathcal{S}_1$ and $\mathcal{S}_2$, whose respective population densities are denoted as $u^\eps(t,x), v^\eps(t,x)$. Note that the integral 
\begin{align*}
\int\limits_{\mathcal{B}} u^\eps(t,x)\, {\rm d}x
\end{align*} 
represents the number of $\mathcal{S}_1$ individuals in a region $\mathcal{B}\subset\Omega$ at any given instant $t>0$. Similar interpretation holds for the population density $v^\eps(t,x)$. For the above two population densities and for a non-negative function $w^\eps(t,x)$, we consider a competition-diffusion system with heterogeneous diffusivities
\begin{equation}\label{eq:model-eps}
\begin{aligned}
\partial_t u^\eps - \nabla \cdot \Big( A^\eps(x) \nabla u^\eps \Big) + \frac{u^\eps}{\eps} \Big( v^\eps + \lambda \left( 1 - w^\eps \right) \Big) & = 0 & \mbox{ in }(0,\ell)\times\Omega,
\\[0.2 cm]
\partial_t v^\eps - \nabla \cdot \Big( B^\eps(x) \nabla v^\eps \Big) + \alpha \frac{v^\eps}{\eps} \Big( u^\eps + \lambda w^\eps \Big) & = 0 & \mbox{ in }(0,\ell)\times\Omega,
\\[0.2 cm]
\partial_t w^\eps + \frac{u^\eps}{\eps} \left( w^\eps - 1 \right) + \frac{w^\eps v^\eps}{\eps} & = 0 & \mbox{ in }(0,\ell)\times\Omega.
\end{aligned}
\end{equation}
The interspecific competition rates between $u^\eps$ and $v^\eps$ in the above evolution is of $\mathcal{O}(\eps^{-1})$. Here the positive parameter $\alpha$ governs the relative competitive strength of the two species, if $\alpha>1$, $\mathcal{S}_1$ has a competitive advantage over $\mathcal{S}_2$ whilst if $\alpha<1$, $\mathcal{S}_2$ has a competitive advantage over $\mathcal{S}_1$. The unknown $w^\eps(t,x)$ may be thought of as an approximation of the characteristic function of the population density for the species $\mathcal{S}_1$, similarly $(1-w^\eps(t,x))$ may be thought of as an approximation of the characteristic function of the population density for the species $\mathcal{S}_2$.  This interpretation allows an ecological interpretation of the reaction terms involving  $w^\eps(t,x)$ as representing the cost to a species of converting the habitat of the other species into its own habitat \cite{Hilhorst_2001}. The parameter $\lambda\geq 0$ governs the relative strength of this cost to direct interspecies competition effects, and we note that our results remain valid in the case $\lambda=0$ which corresponds to no cost of habitat conversion.  
The evolution system \eqref{eq:model-eps} for $(u^\eps,v^\eps,w^\eps)$ is supplemented by initial and boundary conditions
\begin{equation}\label{eq:model-ibvp}
\begin{aligned}
u^\eps(0,x) = u^{\rm in}(x),\, v^\eps(0,x) = v^{\rm in}(x),\, w^\eps(0,x) &= w^{\rm in}(x)  & \mbox{ in }\Omega
\\[0.2 cm]
A^\eps(x) \nabla u^\eps \cdot {\bf n}(x) = B^\eps(x) \nabla v^\eps \cdot {\bf n}(x) & = 0 & \mbox{ on }(0,\ell)\times\partial\Omega.
\end{aligned}
\end{equation}
Here ${\bf n}(x)$ is the unit exterior normal to $\Omega$ at $x\in\partial\Omega$. The initial data are assumed to be non-negative and bounded in $\mathrm L^\infty$, i.e.,
\begin{align}\label{eq:initial-data}
0\le u^{\rm in} \le u_{\rm max} <\infty,
\quad
0\le v^{\rm in} \le v_{\rm max} <\infty,
\quad
0\le w^{\rm in} \le 1.
\end{align}
The competition-diffusion model \eqref{eq:model-eps}-\eqref{eq:model-ibvp} is essentially the evolution model considered in \cite{Hilhorst_2001} except for the high frequency oscillations in the diffusion coefficients.\\
The present work deals with the asymptotic analysis of the coupled system \eqref{eq:model-eps}-\eqref{eq:model-ibvp} in the $\eps\to0$ limit and in the $t\gg1$ regime. The $\eps\to0$ limit procedure corresponds to performing both the homogenization and the strong competition limit simultaneously. The study of strong competition limits for such systems with constant diffusivities are found in \cite{DANCER_1999, Hilhorst_2001}. The novelty of this work is to consider the effect of having highly oscillating diffusivities on the strong competition limit. We employ the method of two-scale convergence to address the periodic homogenization problem. Our main result is Theorem \ref{thm:homogen} which says that the solution family $(u^\eps, v^\eps, w^\eps)$ to \eqref{eq:model-eps}-\eqref{eq:model-ibvp} has a limit point $(u^\ast, v^\ast, w^\ast)$ in certain weak topology. The theorem further characterises the limit point as a solution to a certain Stefan type problem. The study of the long time behaviour of such models is treated numerically where we have performed numerical simulations to make some interesting observations on the so-called competitive Lotka-Volterra system in theoretical ecology. 

The organization of  this paper is as follows. In section \ref{sec:prelim-analysis}, we briefly recall the existence and uniqueness theory for the competition-diffusion model \eqref{eq:model-eps}-\eqref{eq:model-ibvp} and gather some quantitative estimates on the solution family. Section \ref{sec:two-scale} deals with our main result and its proof. The definition of two-scale convergence and associated compactness results are recalled in section \ref{sec:two-scale}. Our numerical results with emphasis on theoretical ecology are given in section \ref{sec:ecology}. Eventually, section \ref{sec:conclude} proposes a more general setting where we consider multiple scales in the competition-diffusion systems. Some interesting observations from the periodic homogenization theory are gathered in Appendix \ref{sec:one-dim-layer-material}.\\

{\bf Acknowledgement: }
The research of the first author was supported by the EPSRC programme grant ``Mathematical fundamentals of Metamaterials for multiscale Physics and Mechanic'' (EP/L024926/1). The second author would like to thank the Isaac Newton Institute for Mathematical Sciences for support and hospitality during the program ``Coupling geometric PDEs with physics for cell morphology, motility and pattern formation'' where this project was initiated. Authors further thank the kind hospitality of the Hausdorff Institute for Mathematics in Bonn during the trimester program on multiscale problems in 2017 where some of this work was in progress. The authors would like to thank Henrik Shahgholian for his fruitful suggestions during the preparation of this article.

%%%%%%%%%%%%%%%%%%%%%%%%%%%%%%%%%%%%%%%%%%%%%%%%%%%%%%%%%%%%%%%%%%%%%%
\section{Existence analysis and preliminary estimates}\label{sec:prelim-analysis}
%%%%%%%%%%%%%%%%%%%%%%%%%%%%%%%%%%%%%%%%%%%%%%%%%%%%%%%%%%%%%%%%%%%%%%
The well-posedness of the initial boundary value problem \eqref{eq:model-eps}-\eqref{eq:model-ibvp} is a bit subtle. We cannot straightaway deduce the existence and uniqueness of solutions from the existence theory of reaction-diffusion systems. The reason being the absence of diffusion for $w^\eps(t,x)$ in the evolution \eqref{eq:model-eps}. We borrow the associated well-posedness result from \cite[Lemma 2.2, page 167]{Hilhorst_2001} (see \cite[section 5, pp.~178--180]{Hilhorst_2001} for a sketch of proof).
\begin{prop}\label{prop:well-posed}
Suppose the initial data satisfy \eqref{eq:initial-data}. There exists a positive time $\ell>0$ such that, for each $\eps>0$, the competition-diffusion system \eqref{eq:model-eps}-\eqref{eq:model-ibvp} possesses a unique classical solution $(u^\eps,v^\eps,w^\eps)$ in $(0,\ell)\times\Omega$.
\end{prop}
\begin{rem}
The result as stated in \cite{Hilhorst_2001} does not speak about the global-in-time existence of solutions to the competition diffusion model. The proof of existence-uniqueness in \cite{Hilhorst_2001} uses the contraction mapping principle. In this approach, they choose a final time $\ell$ -- see \cite[page 180]{Hilhorst_2001}.\\
Even though there is no mention (in \cite{Hilhorst_2001}) of extending the interval of existence in time to $(0,\infty)$, we could argue as in reaction-diffusion theories \cite[Lemma 1.1, page 420]{Pierre_2010}. For example, as we have uniform (in time) $\mathrm L^\infty$ bounds on the solution $(u^\eps,v^\eps,w^\eps)$ -- see Lemma \ref{lem:maxprinc}, we can deduce global-in-time existence as there is no possibility of blow-up in $\mathrm L^\infty$-norm.
\end{rem}
Next, we record a result on a conserved quantity for the competition-diffusion model \eqref{eq:model-eps}-\eqref{eq:model-ibvp}.
\begin{lem}\label{lem:conserve-quantity}
Suppose the initial data satisfy \eqref{eq:initial-data}. Then for each $t\in(0,\infty)$, we have
\begin{align*}
\int\limits_\Omega \left( u^\eps - \frac{v^\eps}{\alpha} + \lambda w^\eps \right)(t,x)\, {\rm d}x 
= \int\limits_\Omega \left( u^{\rm in} - \frac{v^{\rm in}}{\alpha} + \lambda w^{\rm in} \right)(x)\, {\rm d}x
\end{align*}
for each $\eps>0$.
\end{lem}
\begin{proof}
Note that from the competition-diffusion system \eqref{eq:model-eps}, we have
\begin{align*}
\partial_t \left( u^\eps - \frac{v^\eps}{\alpha} + \lambda w^\eps \right)
=
\nabla \cdot \Big( A^\eps(x) \nabla u^\eps \Big)
- \nabla \cdot \Big( B^\eps(x) \nabla v^\eps \Big).
\end{align*}
Integrating the above equality over $\Omega$ and employing the zero-flux boundary conditions from \eqref{eq:model-ibvp}, we arrive at
\begin{align*}
\frac{{\rm d}}{{\rm d}t} \int\limits_\Omega \left( u^\eps - \frac{v^\eps}{\alpha} + \lambda w^\eps \right)(t,x)\, {\rm d}x = 0
\end{align*}
which when integrated in the time variable yields the result.
\end{proof}

We record a result, thanks to the Maximum principle (see \cite[Lemma 2.3, page 168]{Hilhorst_2001}).
\begin{lem}\label{lem:maxprinc}
Suppose the initial data satisfy \eqref{eq:initial-data}. Then the solution $(u^\eps,v^\eps,w^\eps)$ to the competition-diffusion model \eqref{eq:model-eps}-\eqref{eq:model-ibvp} satisfies
\begin{align*}
0\le u^\eps(t,x) \le u_{\rm max},
\quad
0\le v^\eps(t,x) \le v_{\rm max},
\quad
0\le w^\eps(t,x) \le 1
\end{align*}
for each $\eps>0$ and for all $(t,x)\in(0,\ell)\times\Omega$.
\end{lem}
The following result gives a priori bounds on the solutions to \eqref{eq:model-eps} via energy method (see \cite[Lemma 2.4, pp. ~168--169]{Hilhorst_2001} for proof).
\begin{prop}[a priori estimates]\label{prop:apriori}
Suppose the initial data satisfy \eqref{eq:initial-data}. Let $(u^\eps,v^\eps,w^\eps)$ be the solution to the competition model \eqref{eq:model-eps}. Then we have
\begin{align}\label{eq:aprioi-nabla-u-v}
\left\Vert \nabla u^\eps \right\Vert_{\mathrm L^2 ((0,\ell)\times\Omega)} \le C;
\qquad
\left\Vert \nabla v^\eps \right\Vert_{\mathrm L^2 ((0,\ell)\times\Omega)} \le C;
\end{align}
\begin{equation}\label{eq:apriori-reaction}
\begin{aligned}
& \left\Vert u^\eps v^\eps \right\Vert_{\mathrm L^1 ((0,\ell)\times\Omega)} \le C\eps;
\qquad
\left\Vert u^\eps \left( 1 - w^\eps\right) \right\Vert_{\mathrm L^1 ((0,\ell)\times\Omega)} \le C\eps;
\\[0.2 cm]
& \left\Vert v^\eps w^\eps \right\Vert_{\mathrm L^1 ((0,\ell)\times\Omega)} \le C\eps.
\end{aligned}
\end{equation}
\end{prop}
As the impending asymptotic analysis needs to address the nonlinearities in our model, it is essential to obtain strong compactness of our solution family. To that end, we record the following estimates on the time translates (see \cite[Lemma 2.5, pp. ~169--171]{Hilhorst_2001} for proof).
\begin{lem}\label{lem:translates}
Let $(u^\eps,v^\eps,w^\eps)$ be the solution to the competition model \eqref{eq:model-eps}. Then for $\tau>0$,
\begin{equation}
\begin{aligned}
\int\limits_0^{\ell-\tau} \int\limits_\Omega \left\vert u^\eps(t+\tau,x) - u^\eps(t,x) \right\vert^2\, {\rm d}x\, {\rm d}t \le C \tau
\\[0.1 cm]
\int\limits_0^{\ell-\tau} \int\limits_\Omega \left\vert v^\eps(t+\tau,x) - v^\eps(t,x) \right\vert^2\, {\rm d}x\, {\rm d}t \le C \tau.
\end{aligned}
\end{equation}
\end{lem}
Note that the family of population densities $u^\eps$ and $v^\eps$ are uniformly bounded in $\mathrm L^2((0,\ell)\times\Omega)$, i.e.,
\begin{align*}
\left\Vert u^\eps \right\Vert_{\mathrm L^2((0,\ell)\times\Omega)} \le C;
\qquad
\left\Vert v^\eps \right\Vert_{\mathrm L^2((0,\ell)\times\Omega)} \le C
\end{align*}
with the positive constant $C$ being independent of $\eps$, thanks to the uniform $\mathrm L^\infty$ estimates from Lemma \ref{lem:maxprinc}. Combining this with the estimate \eqref{eq:aprioi-nabla-u-v} yields a uniform estimate in $\mathrm L^2((0,\ell);\mathrm H^1(\Omega))$. Note further that Lemma \ref{lem:translates} gives the $\mathrm L^2$ estimates on the time translates of the families $u^\eps$ and $v^\eps$. Hence we can invoke the Aubin-Lions compactness criterion to arrive at the following relative compactness result on the families of the population densities.
\begin{prop}[Relative compactness]\label{prop:relative-compact-L2}
Let $u^\eps(t,x)$ and $v^\eps(t,x)$ be the family of population densities associated with the evolution system \eqref{eq:model-eps}-\eqref{eq:model-ibvp}. Then, up to extraction of subsequence, we have
\begin{align*}
u^\eps  & \relbar\joinrel\relbar\joinrel\relbar\joinrel\relbar\joinrel\to u^* \mbox{ strongly in }\mathrm L^2((0,\ell)\times\Omega),
\\
v^\eps  & \relbar\joinrel\relbar\joinrel\relbar\joinrel\relbar\joinrel\to v^* \mbox{ strongly in }\mathrm L^2((0,\ell)\times\Omega).
\end{align*}
\end{prop}
\begin{rem}
The approach in \cite{Hilhorst_2001} to obtain strong compactness of the family is to treat the translates in both the space time variables and employ the Kolmogorov-Riesz-Fr\'echet criterion in space-time \cite[Theorem 4.26, page 111]{Brezis_2010}. One can also obtain the aforementioned compactness via the criterion mentioned in Simon's seminal paper \cite[Theorem 1, page 71]{Simon_1986}.
\end{rem}

%%%%%%%%%%%%%%%%%%%%%%%%%%%%%%%%%%%%%%%%%%%%%%%%%%%%%%%%%%%%%%%%%%%%%%
%%%%%%%%%%%%%%%%%%%%%%%%%%%%%%%%%%%%%%%%%%%%%%%%%%%%%%%%%%%%%%%%%%%%%%
\section{Two-scale convergence and homogenization}\label{sec:two-scale}
%%%%%%%%%%%%%%%%%%%%%%%%%%%%%%%%%%%%%%%%%%%%%%%%%%%%%%%%%%%%%%%%%%%%%%
%%%%%%%%%%%%%%%%%%%%%%%%%%%%%%%%%%%%%%%%%%%%%%%%%%%%%%%%%%%%%%%%%%%%%%
In this section, we briefly recall the notion of two-scale converegence introduced by Nguetseng \cite{Nguetseng_1989} and further developed by Allaire \cite{Allaire_1992}. This is a notion of multi-scale weak convergence which captures oscillations in a function sequence. To be precise, let us recall the definition. As this paper deals with function sequences that depend on the time variable, we make the choice of presenting the two-scale convergence theory with the time variable. However, the time variable simply plays the role of a parameter. Furthermore, we only give the definition in the $\mathrm L^2$-setting. The theory of two-scale convergence is also available in $\mathrm L^p$-spaces with $p\in(1,\infty)$.
\begin{defi}[Two-scale convergence]\label{defi:two-scale}
A family $f^\eps\subset\mathrm L^2((0,\ell)\times\Omega)$ is said to two-scale converge to a limit $f_0(t,x,y)\in\mathrm L^2((0,\ell)\times\Omega\times Y)$ if the following limit holds for any smooth test function $\psi(t,x,y)$ which is $Y$-periodic in the $y$ variable
\begin{align*}
\lim_{\eps\to0} \iint\limits_{(0,\ell)\times\Omega} f^\eps(t,x) \psi\left(t,x,\frac{x}{\eps}\right)\, {\rm d}x\, {\rm d}t 
=
\iiint\limits_{(0,\ell)\times\Omega\times Y} f_0(t,x,y) \psi(t,x,y)\, {\rm d}y\, {\rm d}x\, {\rm d}t.
\end{align*}
We denote the above convergence as $f^\eps\twoscale f_0$.
\end{defi}
The following compactness result is the cornerstone of the two-scale convergence theory (see \cite[Theorem 1, p.611]{Nguetseng_1989} and \cite[Theorem 1.2, p.1485]{Allaire_1992}).
\begin{thm}[Two-scale compactness]\label{thm:2scale-L2}
Suppose $f^\eps(t,x)$ is a uniformly bounded family in $\mathrm L^2((0,\ell)\times\Omega)$, i.e.,
\begin{align*}
\left\Vert f^\eps \right\Vert_{\mathrm L^2((0,\ell)\times\Omega)} \le C
\end{align*}
with the constant $C>0$ being independent of $\eps$. Then we can extract a subsequence, still denoted $f^\eps(t,x)$, and there exists a limit $f_0(t,x,y)\in\mathrm L^2((0,\ell)\times\Omega\times Y)$ such that
\begin{align*}
f^\eps \twoscale f_0.
\end{align*}
\end{thm}
Another important result in the two-scale convergence theory which makes it a valuable tool in periodic homogenization is as follows.
\begin{thm}[Two-scale compactness in $\mathrm H^1$]\label{thm:2scale-H1}
Suppose $f^\eps(t,x)$ be a uniformly bounded family in $\mathrm L^2((0,\ell);\mathrm H^1(\Omega))$, i.e.,
\begin{align*}
\iint\limits_{(0,\ell)\times\Omega} \left\vert f^\eps(t,x) \right\vert^2\, {\rm d}x\, {\rm d}t
+ \iint\limits_{(0,\ell)\times\Omega} \left\vert \nabla f^\eps(t,x) \right\vert^2\, {\rm d}x\, {\rm d}t
\le C
\end{align*}
with the constant $C>0$ being independent of $\eps$. Then we can extract a subsequence, still denoted $f^\eps(t,x)$, and there exist limits $f_0(t,x)\in\mathrm L^2((0,\ell);\mathrm H^1(\Omega))$, $f_1(t,x,y)\in \mathrm L^2((0,\ell)\times\Omega;\mathrm H^1_\#(Y)/\R)$ such that
\begin{align*}
f^\eps & \relbar\joinrel\relbar\joinrel\relbar\joinrel\relbar\joinrel\rightharpoonup f_0 \mbox{ weakly in }L^2((0,\ell);\mathrm H^1(\Omega))
\\
\nabla f^\eps & \twoscale \nabla_x f_0 + \nabla_y f_1
\end{align*}
\end{thm}
Refer to \cite[Theorem 3, p.618]{Nguetseng_1989} and \cite[Proposition 1.14, p.1491]{Allaire_1992} for the proof of the above result. In the statement of Theorem \ref{thm:2scale-H1}, we have used the standard notation for $Y$-periodic function space
\[
\mathrm H^1_\#(Y) := \Big\{ f:\mathbb{R}^d \to \R \, \, \mbox{ such that }\, \, f\mbox{ is }Y\mbox{-periodic and }\left\Vert f \right\Vert_{\mathrm H^1(Y)} < \infty\Big\}.
\]
The following result states that a norm convergence information on the family implies that the family two-scale converges to the same limit (see \cite[Theorem 5, p.42]{Lukkassen_2002} for the proof).
\begin{thm}[Convergence in norm]
Suppose $f^\eps(t,x)$ is a family which strongly converges in $\mathrm L^2((0,\ell)\times\Omega)$ to $f_0(t,x)$, i.e., 
\begin{align*}
\lim_{\eps\to0}\iint\limits_{(0,\ell)\times\Omega} \left\vert f^\eps (t,x) - f_0(t,x) \right\vert^2\, {\rm d}x\, {\rm d}t = 0
\end{align*}
then we have
\begin{align*}
f^\eps \twoscale f_0
\end{align*}
\end{thm}

%%%%%%%%%%%%%%%%%%%%%%%%%%%%%%%%%%%%%%%%%%%%%%%%%%%%%%%%%%%%%%%%%%%%%%
\subsection{Homogenization result}
%%%%%%%%%%%%%%%%%%%%%%%%%%%%%%%%%%%%%%%%%%%%%%%%%%%%%%%%%%%%%%%%%%%%%%
In this subsection, we state and prove our main result on the $\eps\to0$ limit for the competition-diffusion model. Our strategy of proof is to employ the notion of two-scale convergence detailed earlier.

\begin{thm}[Homogenization limit]\label{thm:homogen}
Let $(u^\eps,v^\eps,w^\eps)$ be the solution to the competition-diffusion model \eqref{eq:model-eps}-\eqref{eq:model-ibvp}. Then there exists a subsequence $(u^\eps, v^\eps, w^\eps)$ and functions $u^*\in\mathrm L^2((0,\ell);\mathrm H^1(\Omega))$, $u_1\in \mathrm L^2((0,\ell)\times\Omega;\mathrm H^1_\#(Y)/\R)$, $v^*\in\mathrm L^2((0,\ell);\mathrm H^1(\Omega))$, $v_1\in \mathrm L^2((0,\ell)\times\Omega;\mathrm H^1_\#(Y)/\R)$, $w^*\in\mathrm L^2((0,\ell)\times\Omega)$ such that
\begin{equation}\label{eq:thm-limits}
\begin{aligned}
u^\eps  & \relbar\joinrel\relbar\joinrel\relbar\joinrel\relbar\joinrel\to u^* \mbox{ strongly in }\mathrm L^2((0,\ell)\times\Omega)
\\
\nabla u^\eps & \twoscale \nabla_x u^* + \nabla_y u_1
\\
v^\eps  & \relbar\joinrel\relbar\joinrel\relbar\joinrel\relbar\joinrel\to v^* \mbox{ strongly in }\mathrm L^2((0,\ell)\times\Omega)
\\
\nabla v^\eps & \twoscale \nabla_x v^* + \nabla_y v_1
\\
w^\eps & \relbar\joinrel\relbar\joinrel\relbar\joinrel\relbar\joinrel\rightharpoonup w^* \mbox{ weakly in }\mathrm L^2((0,\ell)\times\Omega).
\end{aligned}
\end{equation}
Furthermore, we have that the functions $(u_1, v_1)$ can be decomposed as
\begin{equation}\label{eq:thm:u1-v1}
\begin{aligned}
u_1(t,x,y) = \sum_{i=1}^d \omega_i(y) \frac{\partial u^*}{\partial x_i}(t,x);
\qquad
v_1(t,x,y) = \sum_{i=1}^d \chi_i(y) \frac{\partial v^*}{\partial x_i}(t,x)
\end{aligned}
\end{equation}
with $(\omega_i, \chi_i)_{1\le i\le d}$ solving the decoupled system of periodic-boundary value problems
\begin{equation}\label{eq:thm:cell-pbs}
\begin{aligned}
{\rm div}_y \Big( A(y) \Big( {\bf e}_i + \nabla_y \omega_i(y) \Big) \Big) & = 0 \qquad \mbox{ in }Y,
\\[0.1 cm]
{\rm div}_y \Big( B(y) \Big( {\bf e}_i + \nabla_y \chi_i(y) \Big) \Big) & = 0 \qquad \mbox{ in }Y,
\end{aligned}
\end{equation}
for each $i\in\{1,\cdots,d\}$ where $\{{\bf e}_i\}_{i=1}^d$ denotes the canonical basis in $\R^d$.\\
Moreover, the triple of functions $(u^*,v^*,w^*)$ satisfy
\begin{equation}\label{eq:thm:limit-wf}
\begin{aligned}
& - \iint\limits_{(0,\ell)\times\Omega} \left( u^* - \frac{v^*}{\alpha} + \lambda w^* \right) \partial_t \psi(t,x)\, {\rm d}x\, {\rm d}t
\\[0.2 cm]
& - \int\limits_\Omega \left( u^{\rm in}(x) - \frac{v^{\rm in}(x)}{\alpha} + \lambda w^{\rm in}(x) \right) \psi(0,x)  \, {\rm d}x
\\[0.2 cm]
& + \iiint\limits_{(0,\ell)\times\Omega} \Big(  A_{\rm hom} \nabla u^*(t,x) - \frac{B_{\rm hom}}{\alpha} \nabla v^*(t,x) \Big) \cdot \nabla \psi(t,x) \, {\rm d}x\, {\rm d}t
= 0
\end{aligned}
\end{equation}
for any smooth test function $\psi$ such that $\psi(\ell,x)=0$. The homogenized matrices $A_{\rm hom}$ and $B_{\rm hom}$ are given by the formulae in terms of the cell solutions
\begin{equation}\label{eq:thm:Ahom-Bhom}
\begin{aligned}
\left[ A_{\rm hom}\right]_{ij} & = \int\limits_{Y} A(y) \Big( {\bf e}_j + \nabla_y \omega_j(y) \Big)\cdot {\bf e}_i\, {\rm d}y
\\
\left[ B_{\rm hom}\right]_{ij} & = \int\limits_{Y} B(y) \Big( {\bf e}_j + \nabla_y \chi_j(y) \Big)\cdot {\bf e}_i\, {\rm d}y
\end{aligned}
\end{equation}
for $i,j\in\{1,\cdots,d\}$.
\end{thm}
\begin{rem}
The statement of Theorem \ref{thm:homogen} states that the convergences in \eqref{eq:thm-limits} hold only up to extraction of subsequences. It should, however, be noted that the limit Stefan type problem \eqref{eq:thm:limit-wf} has a unique solution -- consult \cite{Hilhorst_2003} for further details. As the limit point is a solution to a problem which is uniquely solvable, this demonstrates that the convergences in \eqref{eq:thm-limits} hold for the entire sequence.
\end{rem}
\begin{proof}[Proof of Theorem \ref{thm:homogen}]
The strong convergence of a subsequence $(u^\eps, v^\eps)$ to $(u^*, v^*)$ in $\mathrm L^2((0,\ell)\times\Omega)$ as given in \eqref{eq:thm-limits} is nothing but the result of Proposition \ref{prop:relative-compact-L2}. With regards to the two-scale limits in \eqref{eq:thm-limits} of the spatial gradients, let us consider the uniform a priori bounds in \eqref{eq:aprioi-nabla-u-v} and the uniform $\mathrm L^\infty$-bounds in Lemma \ref{lem:maxprinc}. Invoking the two-scale compactness result from Theorem \ref{thm:2scale-H1} yields the two-scale limits of the gradients as stated in \eqref{eq:thm-limits}. Finally, with regards to the weak $\mathrm L^2$-limit of the family $w^\eps$, it follows directly from the uniform bound in Lemma \ref{lem:maxprinc}.\\
The strategy of proof to arrive at the limit expression \eqref{eq:thm:limit-wf} is to test the evolution equation in \eqref{eq:model-eps} for $(u^\eps, v^\eps, w^\eps)$ by $(\psi_1^\eps, \alpha^{-1}\psi_2^\eps, \lambda\psi)$ where
\begin{align*}
\psi_1^\eps := \psi(t,x) + \eps \psi_1\left(t,x,\frac{x}{\eps}\right);
\qquad
\psi_2^\eps := \psi(t,x) + \eps \psi_2\left(t,x,\frac{x}{\eps}\right)
\end{align*}
with smooth functions $\psi(t,x), \psi_1(t,x,y), \psi_2(t,x,y)$ which are $Y$-periodic in the $y$ variable and furthermore $(\psi_1^\eps, \psi_2^\eps, \psi)(\ell,x)=0$. Note that the $\mathcal{O}(1)$ test function for all the three unknowns $u^\eps, v^\eps, w^\eps$ is the same, i.e., $\psi(t,x)$. This choice is to exploit the conservation property of the competition-diffusion system \eqref{eq:model-eps}-\eqref{eq:model-ibvp}. The conserved quantity being $u^\eps - \alpha^{-1} v^\eps + \lambda w^\eps$ (see Lemma \ref{lem:conserve-quantity} for further details).\\
The above choice of test functions yields
\begin{equation}\label{eq:wk-form-u}
\begin{aligned}
& - \iint\limits_{(0,\ell)\times\Omega} u^\eps(t,x) \partial_t \psi_1^\eps(t,x)\, {\rm d}x\, {\rm d}t
+ \iint\limits_{(0,\ell)\times\Omega} A^\eps(x) \nabla u^\eps(t,x) \cdot \nabla \psi_1^\eps(t,x)\, {\rm d}x\, {\rm d}t
\\[0.2 cm]
& + \iint\limits_{(0,\ell)\times\Omega} \frac{u^\eps}{\eps} \Big( v^\eps + \lambda \left( 1 - w^\eps \right) \Big) \psi_1^\eps(t,x)\, {\rm d}x\, {\rm d}t
- \int\limits_\Omega u^{\rm in}(x) \psi_1^\eps(0,x)\, {\rm d}x = 0
\end{aligned}
\end{equation}
for the evolution equation for $u^\eps(t,x)$. The evolution equation for $v^\eps(t,x)$ yields
\begin{equation}\label{eq:wk-form-v}
\begin{aligned}
& - \frac{1}{\alpha} \iint\limits_{(0,\ell)\times\Omega} v^\eps(t,x) \partial_t \psi_2^\eps(t,x)\, {\rm d}x\, {\rm d}t
+ \frac{1}{\alpha} \iint\limits_{(0,\ell)\times\Omega} B^\eps(x) \nabla v^\eps(t,x) \cdot \nabla \psi_2^\eps(t,x)\, {\rm d}x\, {\rm d}t
\\[0.2 cm]
& + \iint\limits_{(0,\ell)\times\Omega} \frac{v^\eps}{\eps} \Big( u^\eps + \lambda w^\eps \Big) \psi_2^\eps(t,x)\, {\rm d}x\, {\rm d}t
- \frac{1}{\alpha} \int\limits_\Omega v^{\rm in}(x) \psi_2^\eps(0,x)\, {\rm d}x = 0.
\end{aligned}
\end{equation}
Next, the evolution equation for $w^\eps(t,x)$ yields
\begin{equation}\label{eq:wk-form-w}
\begin{aligned}
& - \lambda \iint\limits_{(0,\ell)\times\Omega} w^\eps(t,x) \partial_t \psi(t,x)\, {\rm d}x\, {\rm d}t
- \lambda \int\limits_\Omega w^{\rm in}(x) \psi(0,x) \, {\rm d}x
\\[0.2 cm]
& + \lambda \iint\limits_{(0,\ell)\times\Omega} \left( \frac{u^\eps}{\eps} \left( w^\eps - 1 \right) + \frac{w^\eps v^\eps}{\eps} \right) \psi(t,x)\, {\rm d}x\, {\rm d}t
= 0. 
\end{aligned}
\end{equation}
Subtracting the expression in \eqref{eq:wk-form-v} from \eqref{eq:wk-form-u} and followed by adding the expression in \eqref{eq:wk-form-w} results in
\begin{equation}\label{eq:wk-form-uvw}
\begin{aligned}
& - \iint\limits_{(0,\ell)\times\Omega} \left( u^\eps - \frac{v^\eps}{\alpha} + \lambda w^\eps \right) \partial_t \psi_1^\eps(t,x)\, {\rm d}x\, {\rm d}t
\\[0.2 cm]
& - \int\limits_\Omega \left( u^{\rm in}(x) \psi_1^\eps(0,x) - \frac{v^{\rm in}(x)}{\alpha} \psi_2^\eps(0,x) + \lambda w^{\rm in}(x) \psi(0,x) \right)  \, {\rm d}x
\\[0.2 cm]
& + \iint\limits_{(0,\ell)\times\Omega} A^\eps(x) \nabla u^\eps(t,x) \cdot \nabla \psi_1^\eps(t,x)\, {\rm d}x\, {\rm d}t
- \frac{1}{\alpha} \iint\limits_{(0,\ell)\times\Omega} B^\eps(x) \nabla v^\eps(t,x) \cdot \nabla \psi_2^\eps(t,x)\, {\rm d}x\, {\rm d}t
\\[0.2 cm]
& + \iint\limits_{(0,\ell)\times\Omega} u^\eps \Big( v^\eps + \lambda \left( 1 - w^\eps \right) \Big) \psi_1\left(t,x,\frac{x}{\eps}\right)\, {\rm d}x\, {\rm d}t
- \iint\limits_{(0,\ell)\times\Omega} v^\eps \Big( u^\eps + \lambda w^\eps \Big) \psi_2\left(t,x,\frac{x}{\eps}\right)\, {\rm d}x\, {\rm d}t = 0.
\end{aligned}
\end{equation}
Note that the integral terms involving the nonlinear terms vanish because
\begin{align*}
& \left\vert\,\, \,
\iint\limits_{(0,\ell)\times\Omega} u^\eps \Big( v^\eps + \lambda \left( 1 - w^\eps \right) \Big) \psi_1\left(t,x,\frac{x}{\eps}\right)\, {\rm d}x\, {\rm d}t
\right\vert
\\[0.2 cm]
& \le \sup_{(t,x)\in(0,\ell)\times\Omega} \psi_1\left(t,x,\frac{x}{\eps}\right)
\iint\limits_{(0,\ell)\times\Omega} u^\eps \Big( v^\eps + \lambda \left( 1 - w^\eps \right) \Big)\, {\rm d}x\, {\rm d}t.
\end{align*}
We have
\begin{align*}
\sup_{(t,x)\in(0,\ell)\times\Omega} \psi_1\left(t,x,\frac{x}{\eps}\right)
=
\sup_{(t,x,y)\in(0,\ell)\times\Omega\times Y} \psi_1(t,x,y) <\infty
\end{align*}
by choice. Hence taking the $\eps\to0$ limit in the previous inequality says that the nonlinear terms do not contribute in the limit, thanks to the a priori estimates \eqref{eq:apriori-reaction} from Proposition \ref{prop:apriori}. Similar argument can be made to show that the other nonlinear term in the expression \eqref{eq:wk-form-uvw} does not contribute to the $\eps\to0$ limit either.\\
Getting back to the expression \eqref{eq:wk-form-uvw}, let us pass to the limit as $\eps\to0$. Using the compactness properties \eqref{eq:thm-limits} proved earlier, we arrive at
\begin{equation}\label{eq:wk-form-two-scale-limit}
\begin{aligned}
& - \iint\limits_{(0,\ell)\times\Omega} \left( u^* - \frac{v^*}{\alpha} + \lambda w^* \right) \partial_t \psi(t,x)\, {\rm d}x\, {\rm d}t
- \int\limits_\Omega \left( u^{\rm in}(x) - \frac{v^{\rm in}(x)}{\alpha} + \lambda w^{\rm in}(x) \right) \psi(0,x)  \, {\rm d}x
\\[0.2 cm]
& + \iiint\limits_{(0,\ell)\times\Omega\times Y} A(y) \Big( \nabla u^*(t,x) + \nabla_y u_1(t,x,y) \Big) \cdot \Big( \nabla \psi(t,x) + \nabla_y \psi_1(t,x,y) \Big)\, {\rm d}y \, {\rm d}x\, {\rm d}t
\\[0.2 cm]
& - \frac{1}{\alpha} \iiint\limits_{(0,\ell)\times\Omega\times Y} B(y) \Big( \nabla v^*(t,x) + \nabla_y v_1(t,x,y) \Big) \cdot \Big( \nabla \psi(t,x) + \nabla_y \psi_2(t,x,y) \Big)\, {\rm d}y \, {\rm d}x\, {\rm d}t
= 0.
\end{aligned}
\end{equation}
As is classical in periodic homogenization via two-scale convergence, while treating the diffusion terms in the expression \eqref{eq:wk-form-uvw}, we consider the products ${A^\eps}^\top\!\nabla \psi_1^\eps$ and ${B^\eps}^\top\!\nabla \psi_2^\eps$ as test functions when passing to the limit using two-scale convergence (see Definition \ref{defi:two-scale}).\\
In \eqref{eq:wk-form-two-scale-limit}, which can be treated as a weak formulation of a coupled two-scale system, let us take $\psi\equiv0$. This results in
\begin{equation*}
\begin{aligned}
& \iiint\limits_{(0,\ell)\times\Omega\times Y} A(y) \Big( \nabla u^*(t,x) + \nabla_y u_1(t,x,y) \Big) \cdot \nabla_y \psi_1(t,x,y) \, {\rm d}y \, {\rm d}x\, {\rm d}t
\\[0.2 cm]
& - \frac{1}{\alpha} \iiint\limits_{(0,\ell)\times\Omega\times Y} B(y) \Big( \nabla v^*(t,x) + \nabla_y v_1(t,x,y) \Big) \cdot \nabla_y \psi_2(t,x,y)\, {\rm d}y \, {\rm d}x\, {\rm d}t
= 0.
\end{aligned}
\end{equation*}
The above expression can be treated as the weak formulation associated with a decoupled system of periodic-boundary value problems in the $y$ variable given below
\begin{equation}
\begin{aligned}
{\rm div}_y \Big( A(y) \Big( \nabla u^*(t,x) + \nabla_y u_1(t,x,y) \Big) \Big) & = 0,
\\[0.1 cm]
{\rm div}_y \Big( B(y) \Big( \nabla v^*(t,x) + \nabla_y v_1(t,x,y) \Big) \Big) & = 0.
\end{aligned}
\end{equation}
The variables $t$ and $x$ are treated as parameters. Note that the linearity of both the equations in the above decoupled system implies that we can solve for the unknowns $u_1$ and $v_1$ with the representation \eqref{eq:thm:u1-v1} with $(\omega_i, \chi_i)_{1\le i\le d}$ solving the decoupled system of periodic-boundary value problems \eqref{eq:thm:cell-pbs}. 

Our next task is to derive the homogenized limit \eqref{eq:thm:limit-wf} for the triple $(u^\ast, v^\ast, w^\ast)$. Hence, in the weak formulation of the coupled two-scale system \eqref{eq:wk-form-two-scale-limit}, take the test functions $\psi_1 = \psi_2 \equiv 0$. This yields
\begin{equation}\label{eq:wk-form-two-scale-limit-bis}
\begin{aligned}
& - \iint\limits_{(0,\ell)\times\Omega} \left( u^* - \frac{v^*}{\alpha} + \lambda w^* \right) \partial_t \psi(t,x)\, {\rm d}x\, {\rm d}t
\\[0.2 cm]
& - \int\limits_\Omega \left( u^{\rm in}(x) - \frac{v^{\rm in}(x)}{\alpha} + \lambda w^{\rm in}(x) \right) \psi(0,x)  \, {\rm d}x
\\[0.2 cm]
& + \iiint\limits_{(0,\ell)\times\Omega\times Y} A(y) \Big( \nabla u^*(t,x) + \nabla_y u_1(t,x,y) \Big) \cdot \nabla \psi(t,x) \, {\rm d}y \, {\rm d}x\, {\rm d}t
\\[0.2 cm]
& - \frac{1}{\alpha} \iiint\limits_{(0,\ell)\times\Omega\times Y} B(y) \Big( \nabla v^*(t,x) + \nabla_y v_1(t,x,y) \Big) \cdot \nabla \psi(t,x) \, {\rm d}y \, {\rm d}x\, {\rm d}t
= 0.
\end{aligned}
\end{equation}
Note that substituting for $u_1(t,x,y)$ using \eqref{eq:thm:u1-v1} in the third line of the above expression yields
\begin{align*}
& \iiint\limits_{(0,\ell)\times\Omega\times Y} A(y) \Big( \nabla u^*(t,x) + \nabla_y u_1(t,x,y) \Big) \cdot \nabla \psi(t,x) \, {\rm d}y \, {\rm d}x\, {\rm d}t
\\
& = \iiint\limits_{(0,\ell)\times\Omega\times Y} A(y) \left( \nabla u^*(t,x) + \nabla_y \left( \sum_{i=1}^d \omega_i(y) \frac{\partial u^*}{\partial x_i}(t,x) \right) \right) \cdot \nabla \psi(t,x) \, {\rm d}y \, {\rm d}x\, {\rm d}t
\\
& = \iint\limits_{(0,\ell)\times\Omega} A_{\rm hom} \nabla u^*(t,x) \cdot \nabla \psi(t,x) \, {\rm d}x\, {\rm d}t
\end{align*}
with the matrix $A_{\rm hom}$ denoting the homogenized matrix whose elements are given by \eqref{eq:thm:Ahom-Bhom}. Similar computations can be done for the term on the fourth line of \eqref{eq:wk-form-two-scale-limit-bis} yielding an expression with the homogenized matrix $B_{\rm hom}$. Hence we arrive at the expression \eqref{eq:thm:limit-wf} for the limit point $(u^*,v^*,w^*)$.
\end{proof}

%%%%%%%%%%%%%%%%%%%%%%%%%%%%%%%%%%%%%%%%%%%%%%%%%%%%%%%%%%%%%%%%%%%%%%
\subsection{Limit problem}
%%%%%%%%%%%%%%%%%%%%%%%%%%%%%%%%%%%%%%%%%%%%%%%%%%%%%%%%%%%%%%%%%%%%%%
In the previous subsection, the solution family $(u^\eps,v^\eps,w^\eps)$ is such that the limit point $(u^*,v^*,w^*)$ solves certain integral expression \eqref{eq:thm:limit-wf}. Here we formulate \eqref{eq:thm:limit-wf} as a weak form of a certain initial-boundary value problem:
\begin{equation}\label{eq:limit:equation-Z-phi-Z}
\begin{aligned}
\partial_t Z - \nabla \cdot \Big( \mathcal{D}(\phi(Z)) \nabla \phi(Z) \Big) = 0
\end{aligned}
\end{equation}
where $\mathcal{D}:\R\to\R^{d\times d}$ defined as
\begin{equation*}
\mathcal{D}(s)
:=
\left\{
\begin{array}{cc}
A_{\rm hom} & \mbox{ for }s>0
\\[0.2 cm]
B_{\rm hom} & \mbox{ for }s<0
\end{array}\right.
\end{equation*}
where $A_{\rm hom}$ and $B_{\rm hom}$ are the homogenized matrices given by \eqref{eq:thm:Ahom-Bhom}.\\
The real-valued function $\phi$ is defined as
\begin{equation*}
\phi(s)
:=
\left\{
\begin{array}{cl}
s & \mbox{ for }s\in (-\infty, 0)
\\[0.2 cm]
0 & \mbox{ for }s\in [0,\lambda]
\\[0.2 cm]
s-\lambda & \mbox{ for }s\in(\lambda,\infty).
\end{array}\right.
\end{equation*}
Now take 
\begin{align*}
Z := u^* - \frac{v^*}{\alpha} + \lambda w^*
\end{align*}
Remark that because of the segregation property of the limit point $(u^*,v^*,w^*)$, we have
\begin{align*}
\phi(Z) = u^* - \frac{v^*}{\alpha}
\end{align*}
and 
\begin{equation*}
\mathcal{D}(\phi(Z))
:=
\left\{
\begin{array}{cc}
A_{\rm hom} & \mbox{ when }u^*>0
\\[0.2 cm]
B_{\rm hom} & \mbox{ when }v^*>0.
\end{array}\right.
\end{equation*}
Rewriting the limit equation \eqref{eq:limit:equation-Z-phi-Z} as a two-phase Stefan problem with a Stefan type condition on the interface between the two segregated species is similar to the calculations done in \cite[Theorem 3.7]{Hilhorst_2001}. For readers' convenience, we recall the details. Let us set
\begin{align*}
\Omega_+(t) & := \Big\{ x\in\Omega \mbox{ such that }\phi(Z(t,x)) > 0\Big\}
\\
\Omega_-(t) & := \Big\{ x\in\Omega \mbox{ such that }\phi(Z(t,x)) < 0\Big\}
\\
\Gamma(t) & := \Omega \setminus \left( \Omega_+(t) \cup \Omega_-(t) \right).
\end{align*}
\begin{thm}[Two-phase Stefan problem formulation]\label{thm:two-phase-formulation}
Let $Z$ be the unique solution to \eqref{eq:limit:equation-Z-phi-Z} for a given initial datum. Suppose $\Gamma(t)$ is smooth and satisfies $\Gamma(t)\cap\partial\Omega =\emptyset$ for all $t\in[0,\ell]$. Let $\nu$ be the unit normal vector on $\Gamma(t)$ oriented from $\Omega_+(t)$ to $\Omega_-(t)$. Assume further that $\Gamma(t)$ moves smoothly with a velocity ${\bf V}(t)$. Then the functions
\[
u:= \phi(Z)_+\qquad \mbox{ and }\quad v:=\alpha\phi(Z)_-
\]
along with $\Gamma(t)$ satisfy
\begin{equation}\label{eq:Two-phase-stefan-form}
\left\{
\begin{aligned}
\partial_t u & = {\rm div}\Big( A_{\rm hom} \nabla u \Big) \qquad \qquad \qquad \quad \mbox{ in }\Omega_+(t)
\\
\partial_t v & = {\rm div}\Big( B_{\rm hom} \nabla v \Big) \qquad \qquad \qquad \quad \mbox{ in }\Omega_-(t)
\\
\lambda {\bf V}(t)\cdot {\bf \nu} & = -A_{\rm hom}\nabla u\cdot {\bf \nu} - B_{\rm hom} \nabla v\cdot {\bf \nu} \quad \mbox{ on }\Gamma(t)
\\
u & = v = 0 \qquad \qquad \qquad \qquad \qquad \mbox{ on }\Gamma(t)
\\
A_{\rm hom} \nabla u \cdot {\bf n} & = B_{\rm hom} \nabla u \cdot {\bf n} = 0 \qquad \quad \qquad \mbox{ on }\partial\Omega
\end{aligned}
\right.
\end{equation}
for $t\in(0,\ell]$.
\end{thm}
\begin{rem}
The Stefan condition in \eqref{eq:Two-phase-stefan-form} says that the normal component of the velocity of the interface is nothing but the difference in contribution of the \emph{fluxes} in the normal direction on to the interface, i.e.,
\begin{align*}
\lambda {\bf V}(t)\cdot {\bf \nu} = -A_{\rm hom}\nabla u^*\cdot {\bf \nu} + B_{\rm hom} \nabla v^*\cdot ({\bf -\nu} ).
\end{align*}
The operative term in the above sentence is \emph{fluxes}. So, it is not just the diffusion coefficient that counts. But the flux in general. In this light, we remark that a recent study \cite{Girardin_2015}  focusses on a similar, albeit simplified, setting to the one considered in this work and concludes that ``unity is not strength'', i.e., that up to a multiplicative constant the more diffusive species invades the habitat of the other. A rigorous proof of similar results in the present setting, i.e., that oscillations in diffusivities hinder the ability of species to invade is beyond the scope of the present work however, in section \ref{sec:ecology} we conduct some numerical simulations that suggest that indeed oscillations are detrimental to invasion and that the direction of motion of the interface may be reversed if the oscillations are sufficiently rapid.
\end{rem}

%%%%%%%%%%%%%%%%%%%%%%%%%%%%%%%%%%%%%%%%%%%%%%%%%%%%%%%%%%%%%%%%%%%%%%
\subsection{Homogenized coefficients}
%%%%%%%%%%%%%%%%%%%%%%%%%%%%%%%%%%%%%%%%%%%%%%%%%%%%%%%%%%%%%%%%%%%%%%
Interesting aspects of the result in Theorem \ref{thm:homogen} are the cell problems \eqref{eq:thm:cell-pbs} and the expressions for the effective coefficients \eqref{eq:thm:Ahom-Bhom}. Below, we state a classical result from the theory of periodic homogenization for the conductivity problem. A proof of this result can be found in any standard text on homogenization (see for e.g., \cite{Bensoussan_1978, Jikov_1994}).
\begin{thm}\label{thm:conduct-homogen}
Let $f^\eps$ be the solution to the following boundary value problem
\begin{equation}\label{eq:conduct-eps-model}
\begin{aligned}
-\nabla \cdot \Big( \mathcal{A}\left(\frac{x}{\eps}\right) \nabla f^\eps \Big) & = g \quad \mbox{ in }\Omega
\\[0.2 cm]
f^\eps & = 0 \quad \mbox{ on }\partial\Omega
\end{aligned}
\end{equation}
with $\mathcal{A}\left(\frac{x}{\eps}\right)$ being an $\eps$-periodic bounded positive-definite matrix and the source term $g\in\mathrm L^2(\Omega)$. Then, the family $f^\eps(x)$ is such that
\begin{equation}
\begin{aligned}
f^\eps & \relbar\joinrel\relbar\joinrel\relbar\joinrel\relbar\joinrel\to f^* \mbox{ strongly in }\mathrm L^2(\Omega)
\\
\nabla f^\eps & \twoscale \nabla_x f^* + \nabla_y f_1
\end{aligned}
\end{equation}
with the corrector $f_1(x,y)$ having the representation
\begin{align*}
f_1(x,y) = \sum_{i=1}^d \eta_i(y) \frac{\partial f^*}{\partial x_i}(x)
\end{align*}
and the $(\eta_i)_{1\le i\le d}$ solving the periodic-boundary value problems
\begin{align}\label{eq:conduct-cell-pb}
\nabla_y \cdot \Big( \mathcal{A}(y) \Big( {\bf e}_i + \nabla_y \eta_i (y) \Big) \Big) = 0 \quad \mbox{ in }Y
\end{align}
for each $i\in\{1,\dots,d\}$. Furthermore the limit point $f^*$ uniquely solves the boundary value problem
\begin{equation}
\begin{aligned}
-\nabla \cdot \Big( \mathcal{A}_{\rm hom}\nabla f^* \Big) & = g \quad \mbox{ in }\Omega
\\[0.2 cm]
f^* & = 0 \quad \mbox{ on }\partial\Omega
\end{aligned}
\end{equation}
with the homogenized coefficient $\mathcal{A}_{\rm hom}$ whose elements are given below
\begin{align}\label{eq:conduct-Ahom}
\left[ \mathcal{A}_{\rm hom}\right]_{ij} = \int\limits_Y \mathcal{A}(y) \Big( {\bf e}_j + \nabla_y \eta_j(y) \Big)\cdot {\bf e}_i\, {\rm d}y
\end{align}
for $i,j\in\{1,\dots,d\}$.
\end{thm}
Remark that if the diffusion coefficient $A(y)$ in Theorem \ref{thm:homogen} equals the conductivity coefficient $\mathcal{A}(y)$ in Theorem \ref{thm:conduct-homogen}, then the periodic-boundary value cell problem for $\omega_i(y)$ in \eqref{eq:thm:cell-pbs} is exactly the same as that for $\eta_i(y)$ in \eqref{eq:conduct-cell-pb} for each $i\in\{1,\dots,d\}$. Furthermore, the homogenized coefficient $A_{\rm hom}$ in \eqref{eq:thm:Ahom-Bhom} matches with the homogenized coefficient $\mathcal{A}_{\rm hom}$ in \eqref{eq:conduct-Ahom}. In Appendix \ref{sec:one-dim-layer-material}, we record some interesting aspects about the homogenized coefficient $\mathcal{A}_{\rm hom}$ in one dimensional and layered material settings.

%%%%%%%%%%%%%%%%%%%%%%%%%%%%%%%%%%%%%%%%%%%%%%%%%%%%%%%%%%%%%%%%%%%%%%
%%%%%%%%%%%%%%%%%%%%%%%%%%%%%%%%%%%%%%%%%%%%%%%%%%%%%%%%%%%%%%%%%%%%%%
\section{Numerical investigations of long time behaviour}\label{sec:ecology}
%%%%%%%%%%%%%%%%%%%%%%%%%%%%%%%%%%%%%%%%%%%%%%%%%%%%%%%%%%%%%%%%%%%%%%
%%%%%%%%%%%%%%%%%%%%%%%%%%%%%%%%%%%%%%%%%%%%%%%%%%%%%%%%%%%%%%%%%%%%%%
To illustrate the effect of high frequency oscillations, we report on the results of some numerical experiments. 
We consider a minor modification to the system  \eqref{eq:model-eps} in which we include logistic growth terms for both the diffusible species, i.e., we consider a system of the following form in $(0,\ell)\times\Omega$:
\begin{equation}\label{eq:model-eps-log}
\begin{aligned}
\partial_t u^\eps - \nabla \cdot \Big( A^\eps(x) \nabla u^\eps \Big) + \frac{u^\eps}{\eps} \Big( v^\eps + \lambda \left( 1 - w^\eps \right) \Big)-ru^\eps(1-u^\eps) & = 0,
\\[0.2 cm]
\partial_t v^\eps - \nabla \cdot \Big( B^\eps(x) \nabla v^\eps \Big) + \alpha \frac{v^\eps}{\eps} \Big( u^\eps + \lambda w^\eps \Big) -rv^\eps(1-v^\eps) & = 0,
\\[0.2 cm]
\partial_t w^\eps + \frac{u^\eps}{\eps} \left( w^\eps - 1 \right) + \frac{w^\eps v^\eps}{\eps} & = 0,
\end{aligned}
\end{equation}
with initial data and boundary conditions as in \eqref{eq:model-ibvp}. We note that the analytical results in the previous sections remain valid for the system \eqref{eq:model-eps-log} and the limiting problem we obtain in the $\eps\to0$ limit corresponds to \eqref{eq:thm:limit-wf} with the additional logistic source terms.  We note that \eqref{eq:model-eps-log} corresponds to a so-called competitive Lotka-Volterra system in which the variables $u^\eps$ and $v^\eps$ are interpreted as the population densities of 
two competing species, the variable $w^\eps$ can also be given an ecological interpretation  in terms of the cost to the species of colonising the others habitat \cite{Hilhorst_2001}. An interesting question from the ecological perspective is the determination of the long-time behaviour of solutions. Girardin and Nadin \cite{Girardin_2015} study the above system in the case of constant diffusion coefficients, zero latent heat and strong competition ($\eps\ll 1, \lambda=0$)  in one spatial dimension and they conclude that, up to a multiplicative constant, the more diffusive species invades the habitat of the less diffusive species.

In this section we present some numerical simulations of \eqref{eq:model-eps-log} that illustrate the influence of oscillations on the long-time behaviour of solutions to \eqref{eq:model-eps-log} in light of our results in the previous sections. To approximate \eqref{eq:model-eps-log} we employ an IMEX Euler method for the time discretisation and piecewise linear finite elements for spatial discretisation \cite{Lakkis2013}. To this end, we construct a triangulation $\mathcal{T}$ of the spatial domain $\Omega$, which divides $\Omega$ into a finite number of non-degenerate and non-overlapping simplices such that the triangulation contains no hanging nodes. We denote by $\mathbb{V}$ the space of all continuous piecewise linear functions on $\mathcal{T}$. Let $N_t$ be a  positive integer, we define the uniform (for simplicity) timestep $\tau=T/N_t$, with $T$ being the end time of simulations. For each $l\in\lbrace0,1,\dots,N_t\rbrace$ we define $t^l := l \tau$. 

Let $\lbrace {\bf x}_i\rbrace_{i=0,\dots,N_h}$ denote the set of vertices of the triangulation $\mathcal{T}$. We look for approximations  $U_h^l:=U_h(t^l)\in\mathbb{V}$, $V_h^l:=V_h(t^l)\in\mathbb{V}$ and  $W_h^l:=W_h(t^l)\in\mathbb{V}$ $l=0,\dots,N_t$ (where we have dropped the $\eps$ for notational ease).  One step of the numerical scheme is as follows: given  $U_h^{l-1}, V_h^{l-1}, W_h^{l-1}\in\mathbb{V}$, find $U_h^{l}$ and $V_h^{l}\in\mathbb{V}$ such that, for all $\Phi\in \mathbb{V}$,
\begin{align*}
\int_\Omega \frac{U_h^{l}\Phi}{\tau}+A^\eps(x)\nabla U_h^{l} \cdot \nabla \Phi + & U_h^{l}\left(\eps^{-1}\left(V_h^{l-1}+\lambda\left(1-W_h^{l-1}\right)\right)-r\left(1-U_h^{l-1}\right)\right)
\\
& = \int_\Omega \frac{U_h^{l-1}\Phi}{\tau}
\end{align*}
and
\begin{align*}
\int_\Omega \frac{V_h^{l}\Phi}{\tau}+B^\eps(x)\nabla V_h^{l}\cdot\nabla \Phi +& V_h^{l}\left(\frac{\alpha}{\eps}\left(U_h^{l-1}+\lambda W_h^{l-1}\right)-r\left(1-V_h^{l-1}\right)\right)
\\
&=\int_\Omega \frac{V_h^{l-1}\Phi}{\tau}.
\end{align*}
To complete the scheme, an ODE for the nodal values of $W_h$ is solved with a semi-implicit Euler method at each node of the triangulation.

The explicit treatment of the nonlinear terms leads to the equations being decoupled, so that only linear systems must be solved at each time-step. We have investigated both Picard linearisation and Newton linearisation; the results remain qualitatively unchanged. Our implementation makes use of the ALBERTA C finite element toolbox \cite{schmidt2005design}. 

%%%%%%%%%%%%%%%%%%%%%%%%%%%%%%%%%%%%%%%%%%%%%%%%%%%%%%%%%%%%%%%%%%%%%%
\subsection{1D Simulations}\label{ssec:testcase-1d}
%%%%%%%%%%%%%%%%%%%%%%%%%%%%%%%%%%%%%%%%%%%%%%%%%%%%%%%%%%%%%%%%%%%%%%
At a first instance, we consider  \eqref{eq:model-eps-log}  in a one-dimensional setting with $\Omega=(0,1)$. We fix $r=50$, $\lambda=1$ and $\alpha=1.1$ thus the species $u^\eps$ enjoys a competitive advantage over the species $v^\eps$.  
For the initial conditions we set 
\[
u^{\rm in}(x)=\chi_{\{x<0.5\}},\, v^{\rm in}(x)=\chi_{\{x\geq0.5\}},\,  w^{\rm in}(x)= u^{\rm in}(x)\quad \mbox{ for }x\in\Omega.
\]
For the simulations we employ a uniform mesh with 16384 degrees of freedom (a fine mesh is needed to resolve the highly oscillatory diffusion coefficient) and a fixed timestep of $10^{-3}$.

Figure \ref{fig:1d_const} shows snapshots of the time evolution in the case of no oscillations in either diffusion coefficient, i.e., $A(y)=B(y)=2$ with  $\eps=10^{-3}$. We observe spatial segregation due to the strong competition and due to its competitive advantage the species $u^\eps$ invades the habitat of the species $v^\eps$.  The long time behaviour being complete colonisation of $\Omega$ by the species $u^\eps$ and the extinction of the species  $v^\eps$.
\begin{figure}[htbp]
{\includegraphics[trim = 0mm 0mm 200mm 0mm,  clip, width=.08\linewidth]{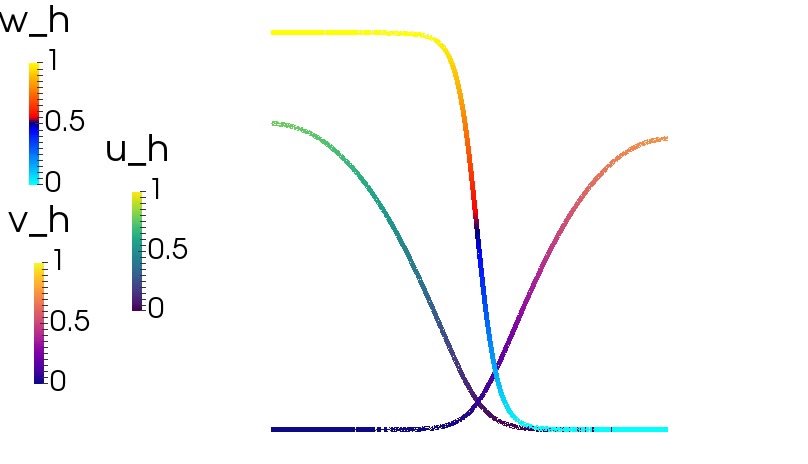}}
\fbox{\includegraphics[trim = 90mm 0mm 50mm 0mm,  clip, width=.15\linewidth]{const_1.jpg}}
\fbox{\includegraphics[trim = 90mm 0mm 50mm 0mm,  clip, width=.15\linewidth]{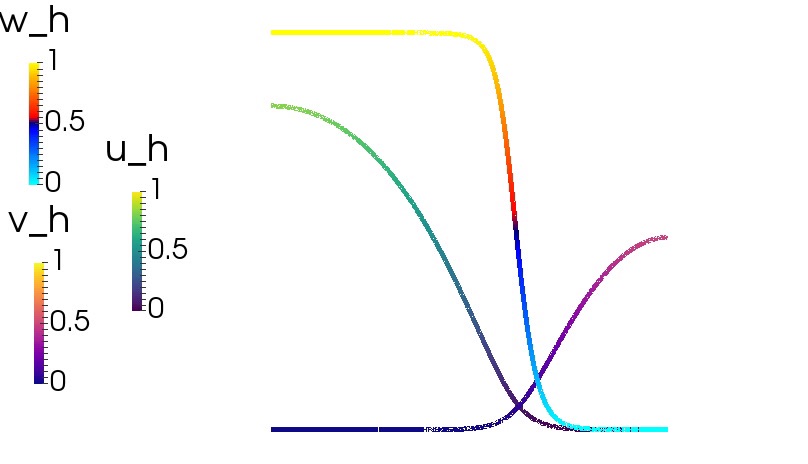}}
\fbox{\includegraphics[trim = 90mm 0mm 50mm 0mm,  clip, width=.15\linewidth]{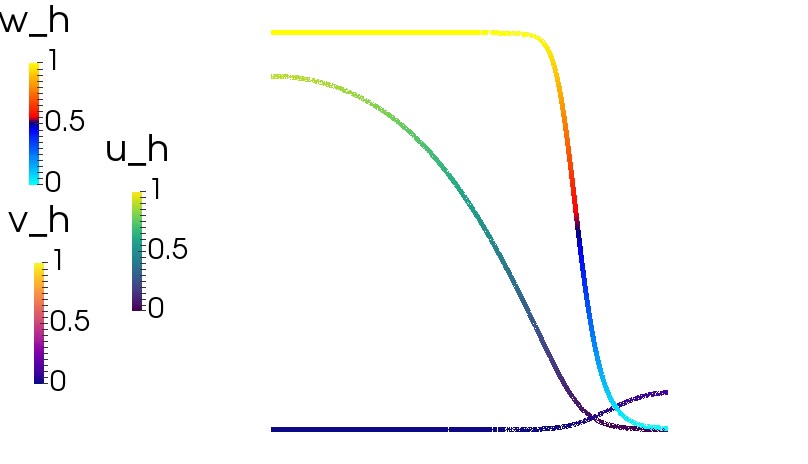}}
\fbox{\includegraphics[trim = 90mm 0mm 50mm 0mm,  clip, width=.15\linewidth]{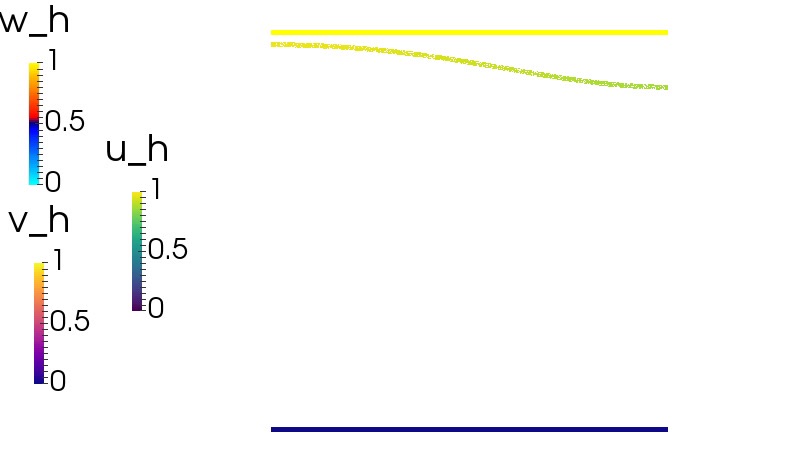}}
\fbox{\includegraphics[trim = 90mm 0mm 50mm 0mm,  clip, width=.15\linewidth]{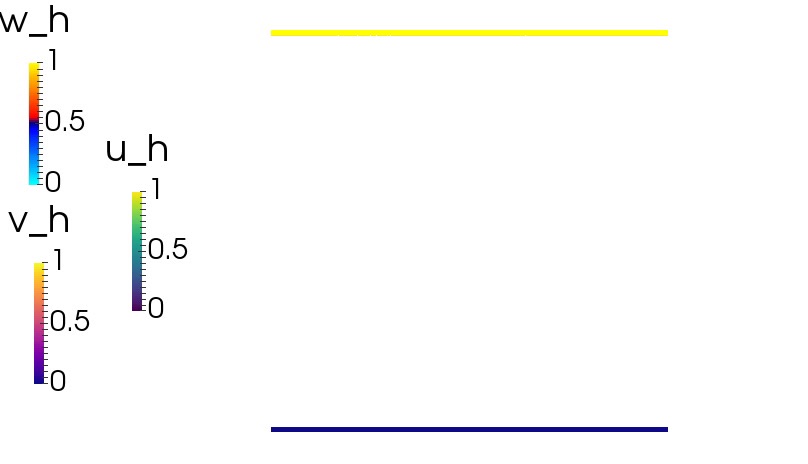}}
\caption{Snapshots of the solution to  \eqref{eq:model-eps-log} with constant diffusion coefficients $A(y)=B(y)=2$ and $\alpha=1.1$ at times $0.1, 0.3, 0.4, 0.5$ and $0.6$ reading from left to right. We observe that due to the competitive advantage species $u^\eps$ invades the habitat of species $v^\eps$.}
\label{fig:1d_const}
\end{figure}

In order to illustrate the effect of high frequency oscillations we now consider the periodic diffusivities associated with the two population densities in \eqref{eq:model-eps-log}  taken to be
\begin{equation}\label{1d_Diff}
\left.
\begin{array}{rl}
A(y) & = 2 + 1.5\sin(2\pi y)
\\[0.2 cm]
B(y) & = 2
\end{array}
\right\}
\quad
\mbox{ for }y\in Y.
\end{equation}
The associated homogenized coefficients in Theorem \ref{thm:homogen} (see in particular \eqref{eq:thm:Ahom-Bhom} and \eqref{eq:1d-homogen-coeff}) are given as
\begin{equation*}
\begin{array}{rl}
A_{\rm hom} & = 1.3229
\\[0.2 cm]
B_{\rm hom} & = 2.
\end{array}
\end{equation*}
Note that the above choice of periodic diffusivities are such that
\begin{equation*}
\begin{array}{cl}
A(y) > B(y) & \mbox{ for }y\in \left(0,\frac{1}{2}\right)
\\[0.2 cm]
A(y) < B(y) & \mbox{ for }y\in \left(\frac{1}{2}, 1\right),
\end{array}
\end{equation*}
i.e., the dominance of one diffusion coefficient over the other inside a periodicity cell is symmetric. But the high frequency oscillations in the homogenization limit lead to coefficients such that $A_{\rm hom} < B_{\rm hom}$ everywhere in the spatial domain. Figure \ref{fig:1d_osc} shows snapshots of the time evolution in the case of diffusivities as in \eqref{1d_Diff} for varying $\eps$. We see that for sufficiently small $\eps$ the long-time behaviour of the system changes from the constant diffusion case 
with the species with constant diffusivity invading the habitat of the species with oscillating diffusivity even though the species with oscillating diffusivity has a competitive advantage over the other.

\begin{figure}[htbp]
{\includegraphics[trim = 0mm 0mm 200mm 0mm,  clip, width=.08\linewidth]{const_1.jpg}}
\fbox{\includegraphics[trim = 60mm 0mm 20mm 0mm,  clip, width=.15\linewidth]{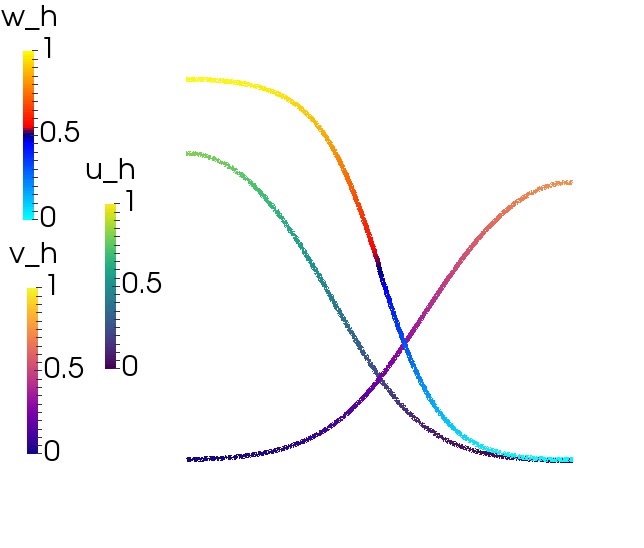}}
\fbox{\includegraphics[trim = 60mm 0mm 20mm 0mm,  clip, width=.15\linewidth]{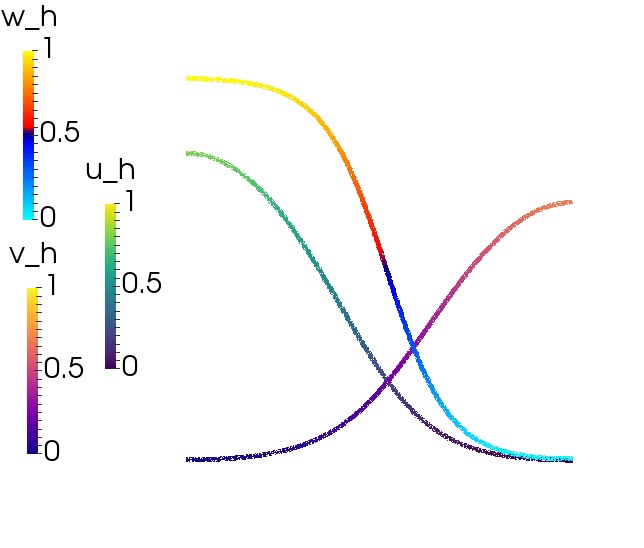}}
\fbox{\includegraphics[trim = 60mm 0mm 20mm 0mm,  clip, width=.15\linewidth]{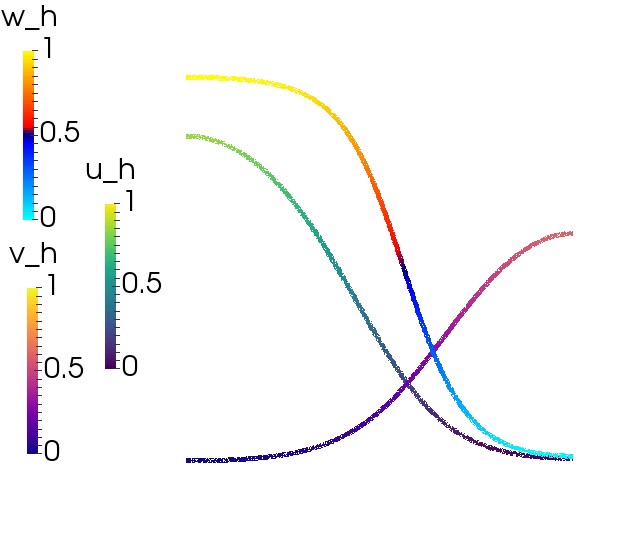}}
\fbox{\includegraphics[trim = 60mm 0mm 20mm 0mm,  clip, width=.15\linewidth]{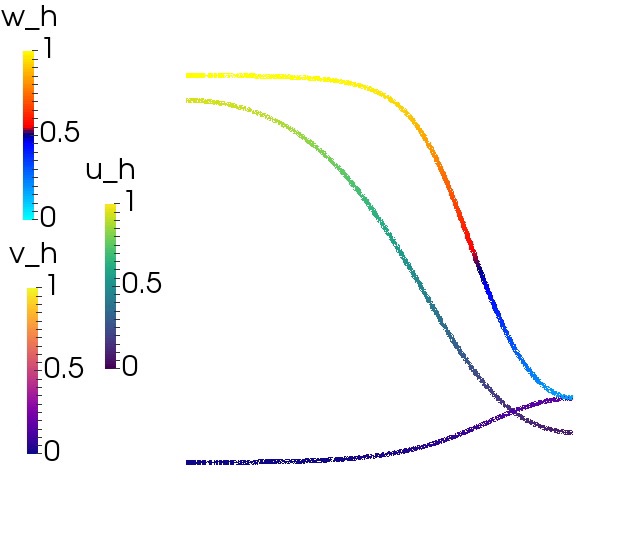}}
\fbox{\includegraphics[trim = 60mm 0mm 20mm 0mm,  clip, width=.15\linewidth]{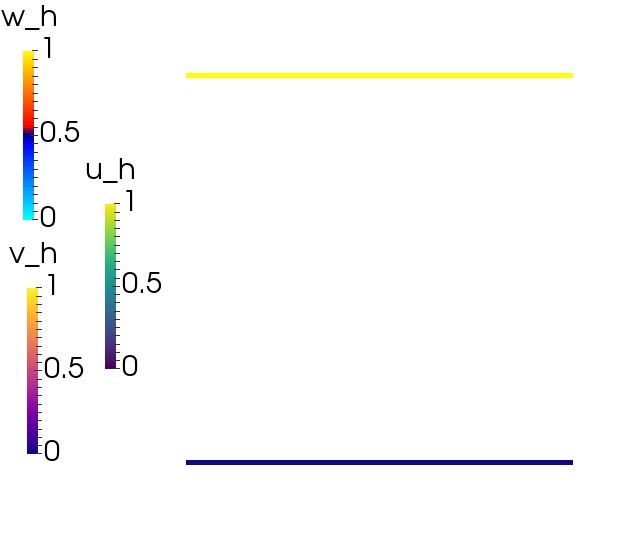}}
{\includegraphics[trim = 0mm 0mm 200mm 0mm,  clip, width=.08\linewidth]{const_1.jpg}}
\fbox{\includegraphics[trim = 60mm 0mm 20mm 0mm,  clip, width=.15\linewidth]{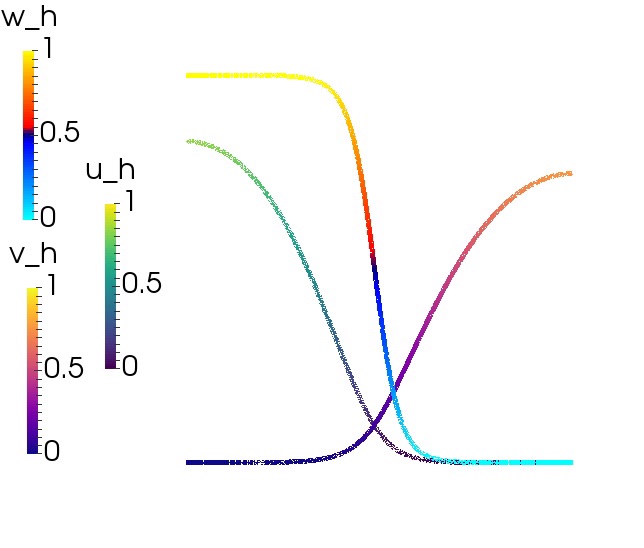}}
\fbox{\includegraphics[trim = 60mm 0mm 20mm 0mm,  clip, width=.15\linewidth]{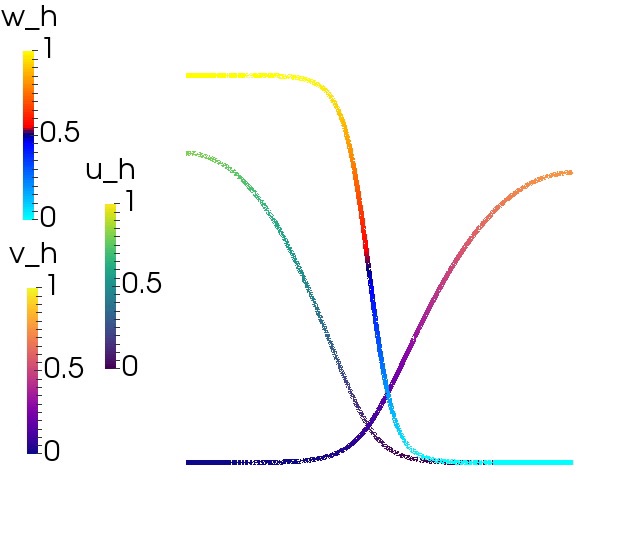}}
\fbox{\includegraphics[trim = 60mm 0mm 20mm 0mm,  clip, width=.15\linewidth]{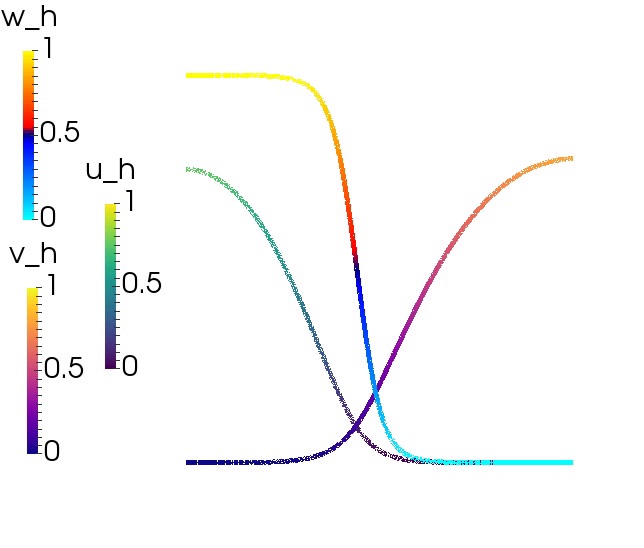}}
\fbox{\includegraphics[trim = 60mm 0mm 20mm 0mm,  clip, width=.15\linewidth]{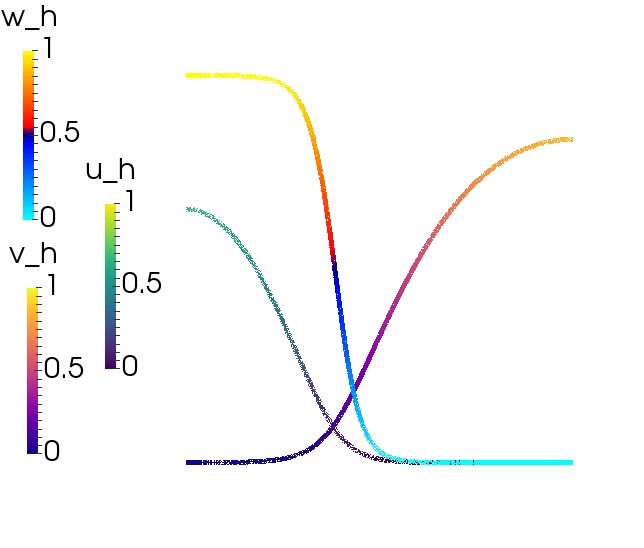}}
\fbox{\includegraphics[trim = 60mm 0mm 20mm 0mm,  clip, width=.15\linewidth]{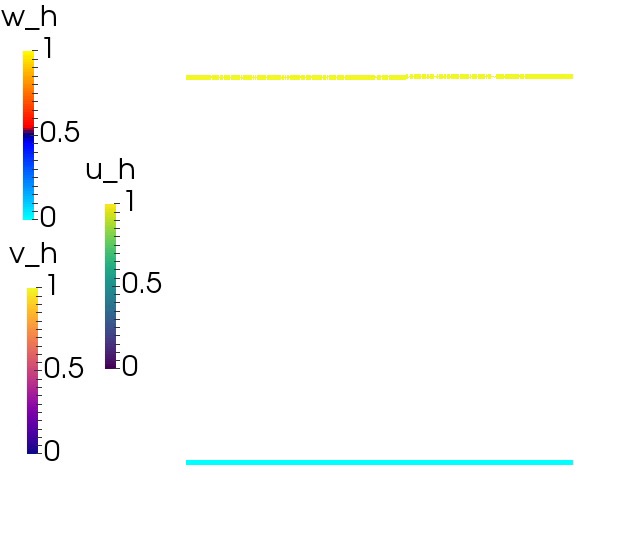}}
{\includegraphics[trim = 0mm 0mm 200mm 0mm,  clip, width=.08\linewidth]{const_1.jpg}}
\fbox{\includegraphics[trim = 60mm 0mm 20mm 0mm,  clip, width=.15\linewidth]{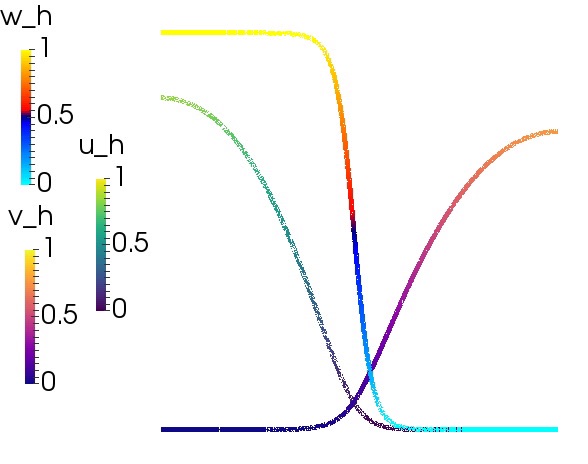}}
\fbox{\includegraphics[trim = 60mm 0mm 20mm 0mm,  clip, width=.15\linewidth]{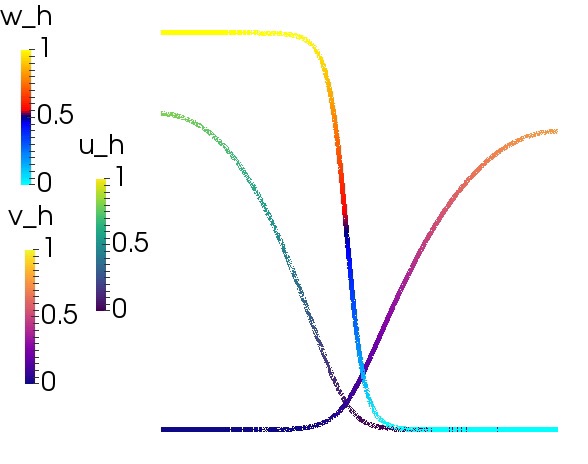}}
\fbox{\includegraphics[trim = 60mm 0mm 20mm 0mm,  clip, width=.15\linewidth]{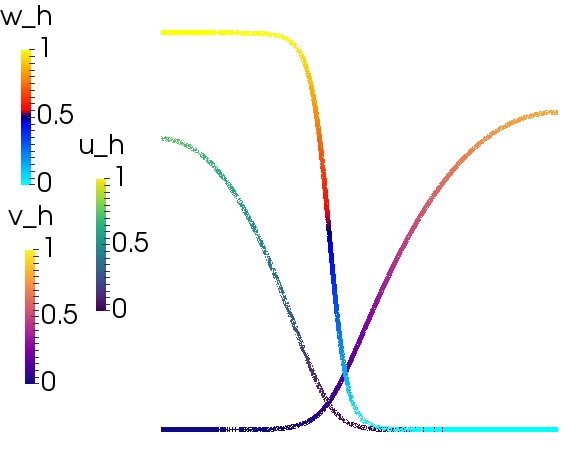}}
\fbox{\includegraphics[trim = 60mm 0mm 20mm 0mm,  clip, width=.15\linewidth]{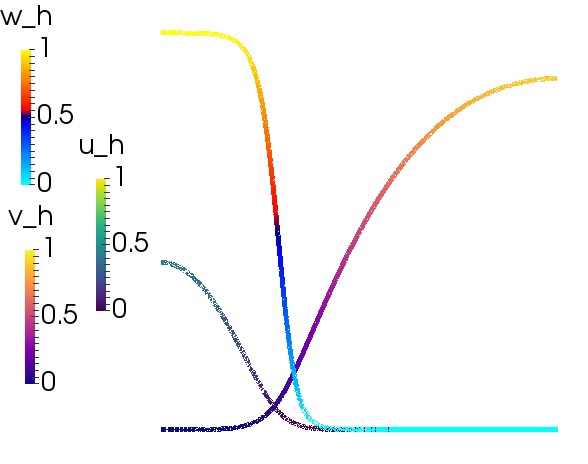}}
\fbox{\includegraphics[trim = 60mm 0mm 20mm 0mm,  clip, width=.15\linewidth]{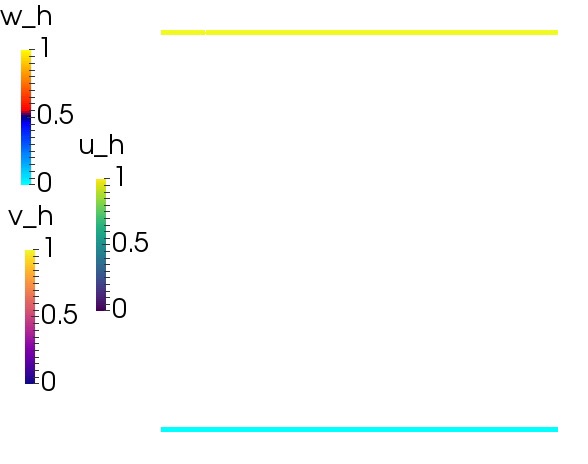}}
\caption{Snapshots of the solution to  \eqref{eq:model-eps-log} with $\eps=10^{-2}$ (top row), $2\times10^{-3}$ (middle row) and $10^{-3}$ (bottom row) at times $0.1, 0.3, 0.4, 0.5$ and $0.6$ reading from left to right in each row. 
The diffusion coefficients are as in \eqref{1d_Diff} and we set $\alpha=1.1$. We observe that despite the competitive advantage , due to the high frequency oscillations, the species $v^\eps$ invades the habitat of species $u^\eps$ for the cases of $\eps=2\times10^{-3}$ and $\eps=10^{-3}$.}
\label{fig:1d_osc}
\end{figure}

%\textcolor{blue}{Notes-HH:} \texttt{The numerical simulation task can be as follows:
%\begin{enumerate}
%\item[$\ast$] Take the one-dimensional domain to be $\Omega:=(0,1)$.
%\item[$\ast$] Fix an initial datum $(u^{\rm in}, v^{\rm in}, w^{\rm in})$ satisfying \eqref{eq:initial-data}.
%\item[$\ast$] Solve the competition-diffusion model \eqref{eq:model-eps}-\eqref{eq:model-ibvp} for the unknown $(u^\eps,v^\eps,w^\eps)$ with the above given initial datum and for four choices of $\eps$ as $1, 0.1, 0.01, 0.001$.
%\item[$\ast$] Compute the solution $(u^*, v^*, w^*)$ to the homogenized limit equation \eqref{eq:thm:limit-wf} with the same above given initial datum.
%\item[$\ast$] Study the behaviour of $(u^\eps,v^\eps,w^\eps)(t)$ and $(u^*, v^*, w^*)(t)$ for large enough time $t\gg1$ -- comparison made with each of the above choice of $\eps$.
%\end{enumerate}
%}
\subsection{2D Simulations}\label{ssec:testcase-2d}
%%%%%%%%%%%%%%%%%%%%%%%%%%%%%%%%%%%%%%%%%%%%%%%%%%%%%%%%%%%%%%%%%%%%%%
Thus far, we treated only the one-dimensional setting. In the example to follow, we consider \eqref{eq:model-eps-log} on a two-dimensional setting with the spatial domain $\Omega:=(0,1)^2$. We fix $\alpha=1.1$, $\lambda=1$ and $r=50$  and start by considering the constant diffusion case, i.e, the periodic diffusivities associated with the species $u^\eps$ and $v^\eps$ are
\[
A(y_1,y_2) =B(y_1,y_2)=2\, \, {\rm Id}.
\]
For the initial conditions we set 
\begin{align*}
u^{\rm in}(x_1,x_2)& =\chi_{\{x_1+0.1\sin(2\pi x_2)<0.5\}}
\\ 
v^{\rm in}(x_1,x_2)& =\chi_{\{x_1+0.1\sin(2\pi x_2)\geq0.5\}}
\\
w^{\rm in}(x)& = u^{\rm in}(x)
\end{align*}
for $x\in\Omega$. For the simulations we employ a uniform mesh with 2099201 degrees of freedom (a fine mesh is needed to resolve the highly oscillatory diffusion coefficient) and a fixed timestep of $10^{-3}$.

Figure \ref{fig:2d_const} shows snapshots of the time evolution in this case. We report on the difference between $u^\eps$ and $v^\eps$ and indicate the zero contour of this difference.   Due to its competitive advantage the species $u^\eps$ invades the habitat of the species $v^\eps$.  The long time behaviour is the complete colonisation of $\Omega$ by the species $u^\eps$ and the extinction of the species  $v^\eps$.
\begin{figure}[htbp]
\fbox{\includegraphics[trim = 0mm 0mm 0mm 0mm,  clip, width=.3\linewidth]{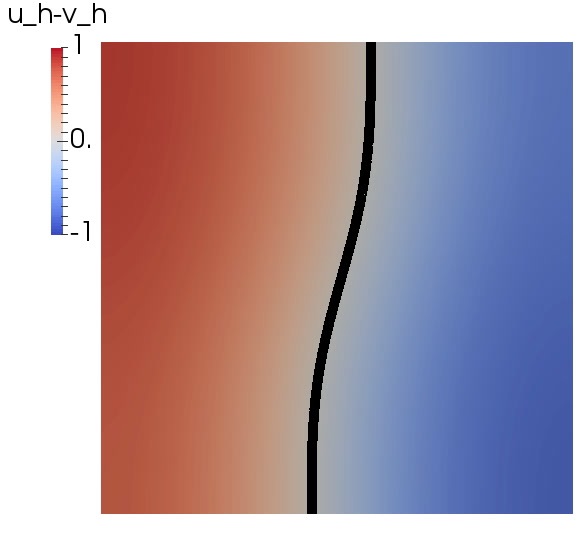}}
\fbox{\includegraphics[trim = 0mm 0mm 0mm 0mm,  clip, width=.3\linewidth]{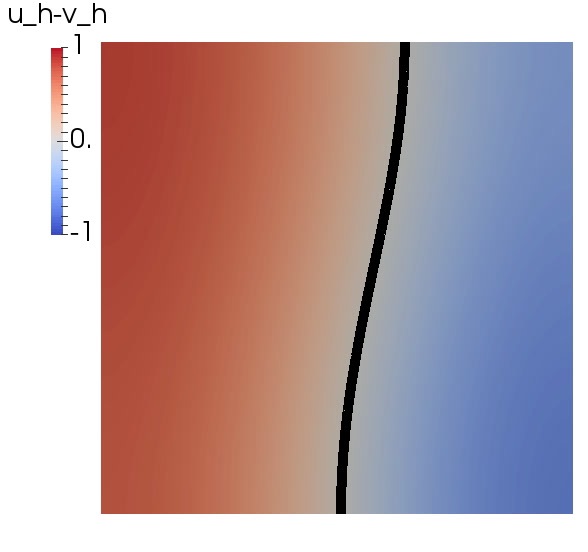}}
\fbox{\includegraphics[trim = 0mm 0mm 0mm 0mm,  clip, width=.3\linewidth]{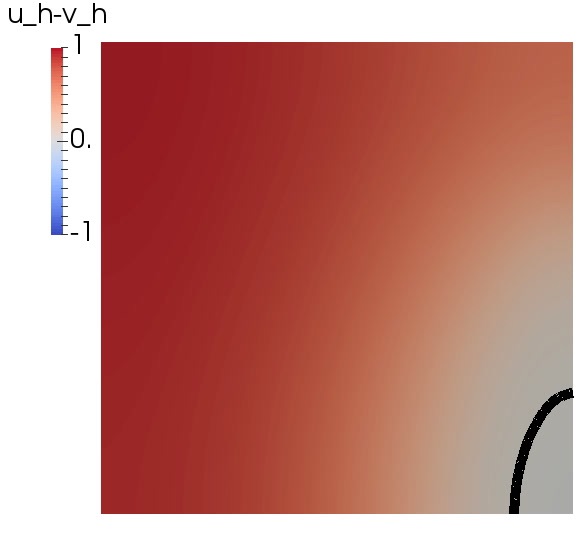}}
\caption{Snapshots of the solution to  \eqref{eq:model-eps-log} with constant diffusivities, $\lambda=1$, $r=50$ and $\alpha=1.1$ at times $0.1, 0.3$ and $0.5$ reading from left to right. We observe that due to the competitive advantage species $u^\eps$ invades the habitat of species $v^\eps$.}
\label{fig:2d_const}
\end{figure}

As above, to illustrate the effect of oscillations, we  now consider the periodic diffusivity associated with $u^\eps$ to be
\begin{align*}
A(y_1,y_2) =
\left(
\begin{matrix}
2 + 1.5\sin(2\pi y_1) & 0
\\[0.2 cm]
0 & 2
\end{matrix}
\right)
\end{align*}
and that associated with the population density $v^\eps$ to be
\begin{align*}
B(y_1,y_2) = 2\, \, {\rm Id}.
\end{align*}
The associated homogenized matrices (using Corollary \ref{cor:2d-conduct-homogen}) are 
\begin{align*}
A_{\rm hom}=
\left(
\begin{matrix}
1.3229 & 0
\\[0.2 cm]
0 & 2
\end{matrix}
\right)\quad \text{ and }\quad
B_{\rm hom} = 2\, \, {\rm Id}.
\end{align*}
Figure \ref{fig:2d_osc} shows snapshots of the time evolution in this case  of oscillating diffusivity with $\eps=2\times10^{-3}$ for one of the species. We see that the behaviour of the system appears to change from the constant diffusion case with the direction of motion of the front between the two species being reversed and the species with constant diffusivity invading the habitat of the species with oscillating diffusivity even though the species with oscillating diffusivity has a competitive advantage over the other.
\begin{figure}[htbp]
\fbox{\includegraphics[trim = 0mm 0mm 0mm 0mm,  clip, width=.3\linewidth]{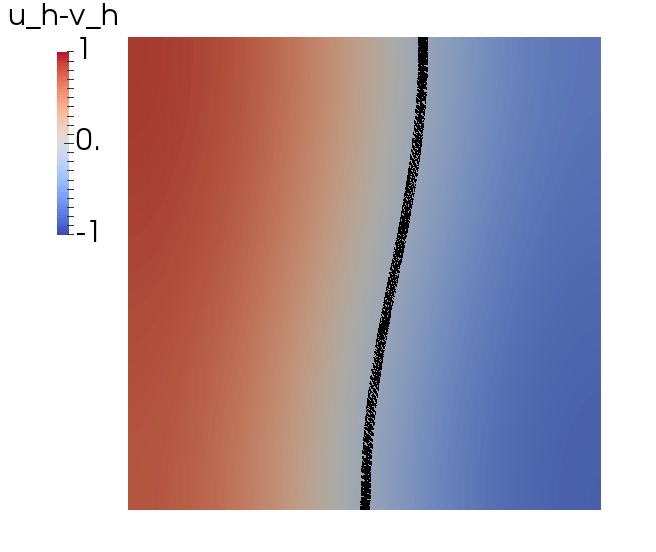}}
\fbox{\includegraphics[trim = 0mm 0mm 0mm 0mm,  clip, width=.3\linewidth]{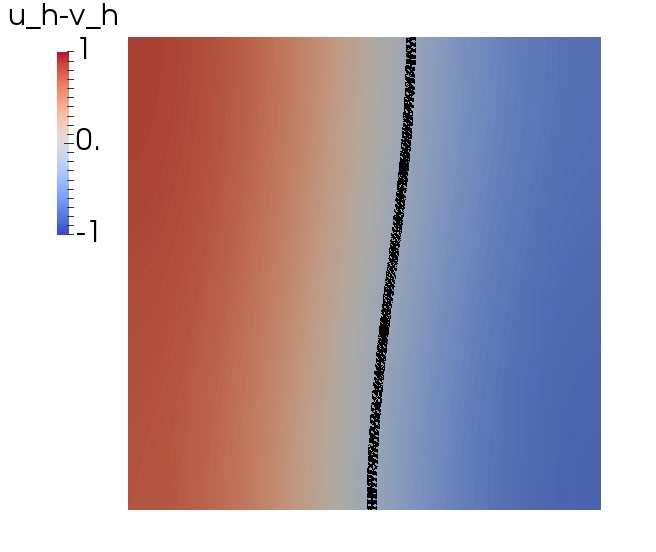}}
\fbox{\includegraphics[trim = 0mm 0mm 0mm 0mm,  clip, width=.3\linewidth]{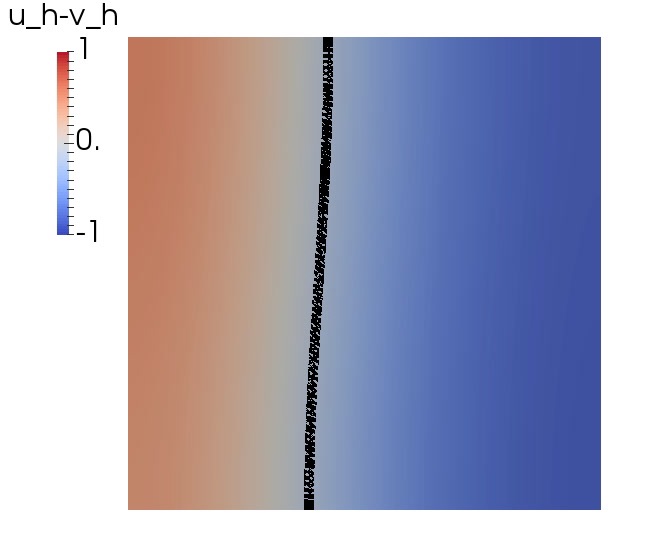}}
\caption{Snapshots of the solution to  \eqref{eq:model-eps-log} with oscillating diffusivity for one species with $\eps=2\times10^{-3}$, $\lambda=1$, $r=50$ and $\alpha=1.1$ at times $0.1, 0.4$ and $1$ reading from left to right. We observe that despite the fact that $u^\eps$ enjoys a competitive advantage over $v^\eps$, due to oscillations in the diffusivity of $u^\eps$, the species $v^\eps$ invades the habitat of species $u^\eps$.}
\label{fig:2d_osc}
\end{figure}

To illustrate this reversal in the invasion behaviour more clearly, in Figure \ref{fig:front}, we plot the position of the front at a series of times in the case of constant diffusivities and in the case of oscillating diffusivity for $u^\eps$ and constant diffusivity for $v^\eps$. The reversal in invasion behaviour is clearly evident.
\begin{figure}
\centering
 \fbox{\includegraphics[trim = 0mm 0mm 0mm 0mm,  clip, width=.45\textwidth]{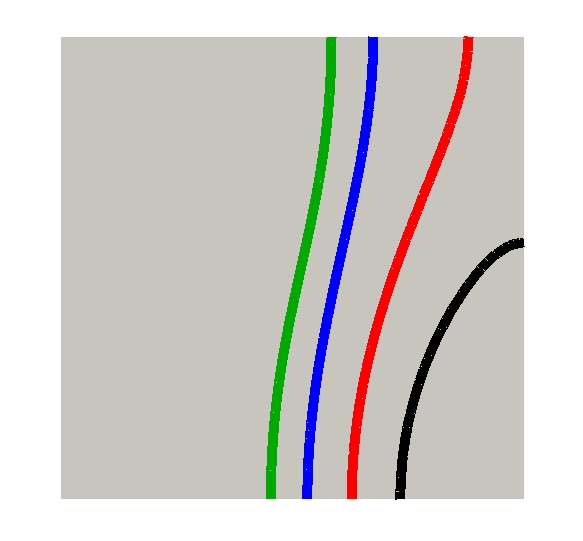}} 
 \hskip 1em
  \fbox{\includegraphics[trim = 0mm 0mm 0mm 0mm,  clip, width=.45\textwidth]{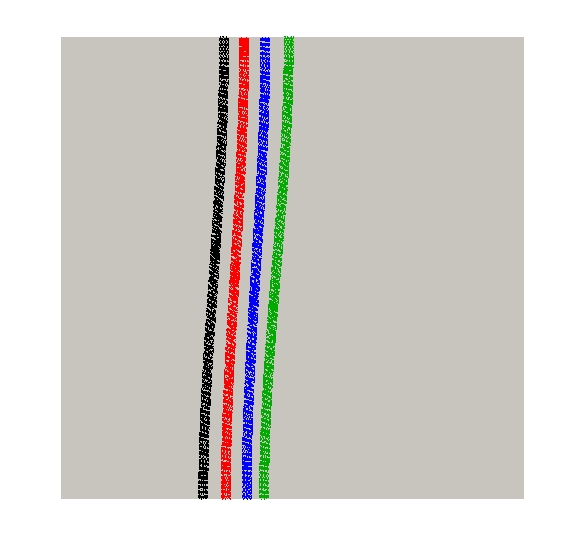}} 
\captionof{figure}{Left: position of the front between the two species habitats ($\{u^\eps=v^\eps\}$),  in the case of constant diffusivities for both species,  at $t=0.1$ (green), $0.3$ (blue), $0.4$ (red) and $0.45$ (black). We observe that the front moves towards the right hand boundary as species $u^\eps$ invades the habitat of $v^\eps$. Right: position of the front between the two species habitats ($\{u^\eps=v^\eps\}$),  in the case of oscillating diffusivity for $u^\eps$ and constant diffusivity for $v^\eps$,  at $t=0.65$ (green), $0.85$ (blue), $0.95$ (red) and $1$ (black). We observe that the direction of motion of the front  is reversed s the species $v^\eps$ invades the habitat of $u^\eps$. }
\label{fig:front}
\end{figure}
\begin{rem}
The aforementioned $2D$ simulations bear some similarity to the simulations in \cite[pp.100--101]{DANCER_1999}. However, the scenario in \cite{DANCER_1999} is to consider isotropic diffusivities in the $\eps$-problem. So, our work is slightly different in that we need to consider anisotropic diffusion and also the work in \cite{DANCER_1999} does not show any simulations to study the $t\gg1$ regime.
\end{rem}

%%%%%%%%%%%%%%%%%%%%%%%%%%%%%%%%%%%%%%%%%%%%%%%%%%%%%%%%%%%%%%%%%%%%%%
%%%%%%%%%%%%%%%%%%%%%%%%%%%%%%%%%%%%%%%%%%%%%%%%%%%%%%%%%%%%%%%%%%%%%%
\section{Different scales for diffusion and competition}\label{sec:conclude}
%%%%%%%%%%%%%%%%%%%%%%%%%%%%%%%%%%%%%%%%%%%%%%%%%%%%%%%%%%%%%%%%%%%%%%
%%%%%%%%%%%%%%%%%%%%%%%%%%%%%%%%%%%%%%%%%%%%%%%%%%%%%%%%%%%%%%%%%%%%%%
We can take a different view in comparison to the calculations presented thus far. In the competition-diffusion model \eqref{eq:model-eps}, the periodic oscillations in the diffusivities are of frequency $\eps$ and the competition rate is of $\mathcal{O}(\eps^{-1})$. To present quite a general scenario, let $\eps$ denote the characteristic length of the inhomogeneities and let $\delta$ denote a parameter such that the competition rate is of $\mathcal{O}(\delta^{-1})$. As we are considering two parameters, we denote the population densities $u^{\eps,\delta}(t,x)$ and $v^{\eps,\delta}(t,x)$ which satisfy the competition-diffusion model
\begin{equation}\label{eq:model-eps-delta}
\begin{aligned}
\partial_t u^{\eps,\delta} - \nabla \cdot \Big( A^\eps(x) \nabla u^{\eps,\delta} \Big) + \frac{u^{\eps,\delta}}{\delta} \Big( v^{\eps,\delta} + \lambda \left( 1 - w^{\eps,\delta} \right) \Big) & = 0 & \mbox{ in }(0,\ell)\times\Omega,
\\[0.2 cm]
\partial_t v^{\eps,\delta} - \nabla \cdot \Big( B^\eps(x) \nabla v^{\eps,\delta} \Big) + \alpha \frac{v^{\eps,\delta}}{\delta} \Big( u^{\eps,\delta} + \lambda w^{\eps,\delta} \Big) & = 0 & \mbox{ in }(0,\ell)\times\Omega,
\\[0.2 cm]
\partial_t w^{\eps,\delta} + \frac{u^{\eps,\delta}}{\delta} \left( w^{\eps,\delta} - 1 \right) + \frac{w^{\eps,\delta} v^{\eps,\delta}}{\delta} & = 0 & \mbox{ in }(0,\ell)\times\Omega.
\end{aligned}
\end{equation}
The evolution system \eqref{eq:model-eps-delta} for $\left(u^{\eps,\delta},v^{\eps,\delta},w^{\eps,\delta}\right)$ is supplemented by initial and boundary conditions
\begin{equation}\label{eq:model-eps-delta-ibvp}
\begin{aligned}
u^{\eps,\delta}(0,x) = u^{\rm in}(x),\, v^{\eps,\delta}(0,x) = v^{\rm in}(x),\, w^{\eps,\delta}(0,x) = w^{\rm in}(x) & & \mbox{ in }\Omega
\\[0.2 cm]
A^\eps(x) \nabla u^{\eps,\delta} \cdot {\bf n}(x) = B^\eps(x) \nabla v^{\eps,\delta} \cdot {\bf n}(x) & = 0 & \mbox{ on }(0,\ell)\times\partial\Omega.
\end{aligned}
\end{equation}
The calculations presented so far treat the case $\eps=\delta$. In this section, we comment on two other cases: (i) $\eps \ll \delta$ and (ii) $\delta \ll \eps$.\\
In the regime $\eps \ll \delta$, for any fixed $\delta>0$, the $\eps\to0$ limit takes over, i.e., the homogenization limit dominates the competition limit. More precisely, fixing a $\delta>0$, the family $\left(u^{\eps,\delta}, v^{\eps,\delta}, w^{\eps,\delta}\right)$ has a limit point $\left(u^{\ast,\delta}, v^{\ast,\delta}, w^{\ast,\delta}\right)$ satisfying
\begin{equation}\label{eq:model-ast-delta}
\begin{aligned}
\partial_t u^{\ast,\delta} - \nabla \cdot \Big( A_{\rm hom} \nabla u^{\ast,\delta} \Big) + \frac{u^{\ast,\delta}}{\delta} \Big( v^{\ast,\delta} + \lambda \left( 1 - w^{\ast,\delta} \right) \Big) & = 0 & \mbox{ in }(0,\ell)\times\Omega,
\\[0.2 cm]
\partial_t v^{\ast,\delta} - \nabla \cdot \Big( B_{\rm hom} \nabla v^{\ast,\delta} \Big) + \alpha \frac{v^{\ast,\delta}}{\delta} \Big( u^{\ast,\delta} + \lambda w^{\ast,\delta} \Big) & = 0 & \mbox{ in }(0,\ell)\times\Omega,
\\[0.2 cm]
\partial_t w^{\ast,\delta} + \frac{u^{\ast,\delta}}{\delta} \left( w^{\ast,\delta} - 1 \right) + \frac{w^{\ast,\delta} v^{\ast,\delta}}{\delta} & = 0 & \mbox{ in }(0,\ell)\times\Omega.
\end{aligned}
\end{equation}
The homogenization procedure is exactly similar to the one present in the proof of Theorem \ref{thm:homogen}. The main ingredient being the compactness properties of the family $\left(u^{\eps,\delta}, v^{\eps,\delta}, w^{\eps,\delta}\right)$ as given in \eqref{eq:thm-limits}. The limit equation obtained above for $\left(u^{\ast,\delta}, v^{\ast,\delta}, w^{\ast,\delta}\right)$ is treated similar to the one treated in \cite{Hilhorst_2001} in the $\delta\to0$ limit which corresponds to the strong competition limit. Finally, the obtained limit is exactly the same as the limit equation \eqref{eq:thm:limit-wf} as in Theorem \ref{thm:homogen}.\\
In the regime $\delta \ll \eps$, the roles are reversed, i.e., for any fixed $\eps>0$, the $\delta\to0$ limit takes over, i.e., the competition limit dominates homogenization limit. More precisely, fixing a $\eps>0$, the family $\left(u^{\eps,\delta}, v^{\eps,\delta}, w^{\eps,\delta}\right)$ has a limit point $\left(u^{\eps,\ast}, v^{\eps, \ast}, w^{\eps,\ast}\right)$ satisfying a two-phase Stefan problem with oscillating diffusivities $A^\eps(x)$ and $B^\eps(x)$. The homogenization of two-phase and single-phase Stefan-type equations has received a lot of attention. Here we cite a few which can be used to homogenize our limit equation for $\left(u^{\eps,\ast}, v^{\eps, \ast}, w^{\eps,\ast}\right)$. Here are the references: \cite{Rodrigues_1982, Visintin_2007, Kim_2008, Kim_2010}.

In full generality, one could consider the two scaling parameters $\eps$ and $\delta$ to be related via a relation $\delta=\eps^\beta$ with $\beta\in(0,\infty)$. It could be of interest, both from the mathematical and applications perspective, to check the feasibility of arriving a result similar in flavour to Theorem \ref{thm:homogen} for a range of values of the exponent $\beta$. Here we cite \cite{Bardos_2016} where a similar question was addressed in the context of simultaneous diffusion and homogenization approximation for the linear Boltzmann equation from kinetic theory. The present scenario, however, is a bit subtle as our competition-diffusion model is nonlinear and as there is species segregation in the limit problem. This analysis in a more general setting is quite intricate and is left for future investigations.

%%%%%%%%%%%%%%%%%%%%%%%%%%%%%%%%%%%%%%%%%%%%%%%%%%%%%%%%%%%%%%%%%%%%%%
\appendix
%%%%%%%%%%%%%%%%%%%%%%%%%%%%%%%%%%%%%%%%%%%%%%%%%%%%%%%%%%%%%%%%%%%%%%

%%%%%%%%%%%%%%%%%%%%%%%%%%%%%%%%%%%%%%%%%%%%%%%%%%%%%%%%%%%%%%%%%%%%%%
\section{One dimensional and layered materials}\label{sec:one-dim-layer-material}
%%%%%%%%%%%%%%%%%%%%%%%%%%%%%%%%%%%%%%%%%%%%%%%%%%%%%%%%%%%%%%%%%%%%%%
As an illustrative example, let us consider the particular one-dimensional setting, i.e., when $d=1$ and when the reference periodicity cell $Y:=[0,1)$. Let us denote the periodic one-dimensional conductivity coefficient as $a(y)$ and the associated homogenized coefficient as $a_{\rm hom}$. In this scenario, the cell problem \eqref{eq:conduct-cell-pb} is to solve for a $1$-periodic function $\eta(y)$ such that
\begin{align}\label{eq:1d-cell-pb}
\frac{\rm d}{{\rm d}y} \Big( a(y) \Big( 1 + \frac{{\rm d}\eta(y)}{{\rm d}y} \Big) \Big) = 0 \quad \mbox{ in }(0,1).
\end{align}
Integrating the differential equation \eqref{eq:1d-cell-pb} yields
\begin{align}\label{eq:1d-integrate-cell-pb}
a(y) \Big( 1 + \frac{{\rm d}\eta(y)}{{\rm d}y} \Big) = \mathsf{c}_1
\end{align}
with $\mathsf{c}_1$ being the integrating constant. Integrating \eqref{eq:1d-integrate-cell-pb} over $(0,y)$ yields
\begin{align*}
\Big( y + \eta(y) \Big) = \mathsf{c}_2 + \int\limits_0^y \frac{\mathsf{c}_1}{a(y')}\, {\rm d}y'
\end{align*}
for another integrating constant $\mathsf{c}_2$. Using the periodic-boundary condition on $\eta(y)$, we determine the constant $\mathsf{c}_1$ as
\begin{align*}
\mathsf{c}_1 \int\limits_0^1 \frac{{\rm d}y}{a(y)} = 1.
\end{align*}
The expression for the homogenized coefficient in \eqref{eq:conduct-Ahom} becomes
\begin{align*}
a_{\rm hom} = \int\limits_0^1 a(y) \Big( 1 + \frac{{\rm d}\eta(y)}{{\rm d}y} \Big)\, {\rm d}y.
\end{align*}
The observation \eqref{eq:1d-integrate-cell-pb} straightaway implies that
\begin{align}\label{eq:1d-homogen-coeff}
a_{\rm hom} \int\limits_0^1 \frac{{\rm d}y}{a(y)} = 1,
\end{align}
i.e., the homogenized coefficient $a_{\rm hom}$ is nothing but the harmonic average of the conductivity coefficient $a(y)$ over the period.\\
As the above calculations suggest, one can give an explicit analytical expression for the homogenized coefficient in one-dimensional periodic settings. Under some specific assumptions (such as the layered materials), explicit analytical expressions can be derived for the associated homogenized coefficient (for further details, see for e.g., \cite[Chapter 1]{Allaire_2002}, \cite[Chapter 5]{Cioranescu_1999} \& \cite[Chapter 12]{Pavliotis-Stuart_2008} and references therein). For readers convenience, we gather a result on layered materials in two dimensions (see \cite[pp.~193--195]{Pavliotis-Stuart_2008} for details). 
\begin{cor}\label{cor:2d-conduct-homogen}
Take the periodicity reference cell $Y:=[0,1)^2$ and take the periodic diffusivity $\mathcal{A}(y)$ in Theorem \ref{thm:conduct-homogen} to be
\begin{align*}
\mathcal{A}(y_1,y_2) = 
\left(
\begin{matrix}
\mathcal{A}_{11}(y_1) & \mathcal{A}_{12}(y_1)
\\[0.2 cm]
\mathcal{A}_{21}(y_1) & \mathcal{A}_{22}(y_1)
\end{matrix}
\right).
\end{align*}
The associated homogenized matrix $\mathcal{A}_{\rm hom}$ is given below
\begin{align*}
\left[ \mathcal{A}_{\rm hom}\right]_{11} \int\limits_0^1 \frac{{\rm d}y_1}{\mathcal{A}_{11}(y_1)} = 1;
\qquad
\left[ \mathcal{A}_{\rm hom}\right]_{12} \int\limits_0^1 \frac{{\rm d}y_1}{\mathcal{A}_{11}(y_1)} = \int\limits_0^1 \frac{\mathcal{A}_{12}(y_1)}{\mathcal{A}_{11}(y_1)} \, {\rm d}y_1;
\end{align*}
\begin{align*}
\left[ \mathcal{A}_{\rm hom}\right]_{21} \int\limits_0^1 \frac{{\rm d}y_1}{\mathcal{A}_{11}(y_1)} & = \int\limits_0^1 \frac{\mathcal{A}_{21}(y_1)}{\mathcal{A}_{11}(y_1)} \, {\rm d}y_1;
\end{align*}
\begin{align*}
\left[ \mathcal{A}_{\rm hom}\right]_{22} \int\limits_0^1 \frac{{\rm d}y_1}{\mathcal{A}_{11}(y_1)}  
= & \left( \int\limits_0^1 \frac{\mathcal{A}_{21}(y_1)}{\mathcal{A}_{11}(y_1)} \, {\rm d}y_1\right) \left( \int\limits_0^1 \frac{\mathcal{A}_{12}(y_1)}{\mathcal{A}_{11}(y_1)} \, {\rm d}y_1\right)
\\[0.2 cm]
& + \left( \int\limits_0^1 \frac{{\rm d}y_1}{\mathcal{A}_{11}(y_1)} \right) \left( \int\limits_0^1 \left( \mathcal{A}_{22}(y_1) - \frac{\mathcal{A}_{12}(y_1)\mathcal{A}_{21}(y_1)}{\mathcal{A}_{11}(y_1)}\right)\, {\rm d}y_1 \right).
\end{align*}
\end{cor}

%%%%%%%%%%%%%%%%%%%%%%%%%%%%%%%%%%%%%%%%%%%%%%%%%%%%%%%%%%%%%%%%%%%%%%
%%%%%%%%%%%%%%%%%%%%%%%%%%%%%%%%%%%%%%%%%%%%%%%%%%%%%%%%%%%%%%%%%%%%%%
%%%%%%%%%%%%%%%%%%%%%%%%%%%%%%%%%%%%%%%%%%%%%%%%%%%%%%%%%%%%%%%%%%%%%%
\bibliographystyle{alpha}\bibliography{biblio}
%%%%%%%%%%%%%%%%%%%%%%%%%%%%%%%%%%%%%%%%%%%%%%%%%%%%%%%%%%%%%%%%%%%%%%
%%%%%%%%%%%%%%%%%%%%%%%%%%%%%%%%%%%%%%%%%%%%%%%%%%%%%%%%%%%%%%%%%%%%%%
\end{document}